\documentclass[11pt]{article}
\usepackage{amsfonts}
\usepackage{amsmath}
\usepackage{amsthm}
\usepackage{amssymb}
\usepackage{mathrsfs}
\usepackage[numbers]{natbib}
\usepackage[fit]{truncate}
\usepackage{fullpage}
\usepackage{mathtools}
\usepackage{makecell}
\usepackage{subfig}
\usepackage{hyperref}
\usepackage{dsfont}

\newcommand{\N}{\mathbb{N}}
\newcommand{\R}{\mathbb{R}}
\renewcommand{\P}{\mathbb{P}}
\newcommand{\E}{\mathbb{E}}

\theoremstyle{plain}
\newtheorem{theorem}{Theorem}[section]

\newtheorem{lemma}[theorem]{Lemma}

\newtheorem{proposition}[theorem]{Proposition}

\theoremstyle{definition}

\begin{document}

\title{Wages and Capital returns in a generalized Pólya urn}
\author{Thomas Gottfried$^*$ and Stefan Großkinsky\footnote{Stochastics and Its Applications, Institute of Mathematics, University of Augsburg}}
\date{\today}
\maketitle

\begin{abstract}

It is a widely observed phenomenon that wealth is distributed significantly more unequal than wages. In this paper we study this phenomenon using a new extension of Pólya's urn, modelling wealth growth through wages and capital returns. We focus in particular on the role of increasing return rates on capital, which have been identified as a main driver of inequality, and labor share, the second main parameter of our model. We fit the parameters from real-world data in Germany, so that simulation results reproduce the empirical wealth distribution and recent dynamics in Germany quite accurately, and are essentially independent from initial conditions. 
Our model is simple enough to allow for a detailed mathematical analysis and provides interesting predictions for future developments and on the importance of wages and capital returns for wealth aggregation. 
We also provide an extensive  discussion of the robustness of our results and the plausibility of the main assumptions used in our model.
\end{abstract}

\section{Introduction}

The evolution of inequality and its determinants is a much discussed issue in research and public debates, not least due to the enormous public impact of Thomas Piketty's work \cite{Piketty, Piketty2}. Most studies consent (e.g. \cite{forbes, quadrini, gabaix,  benhabib, chatterjee, bouchaud, wold, chakrabarti, druagulescu}) that in industrial countries the distribution of wealth reveals a two-tailed structure: Whereas for the majority of the population (95-99\%) the empirical distribution can be well-described by a light-tailed or log-normal distribution, a power-law distribution turns out to be more suitable for the richest within an economy. In fact, the Pareto-distribution was initially suggested by Vilfredo Pareto \cite{pareto} in 1897 to describe the distribution of income and wealth. Moreover, wealth is distributed significantly more unequal than income (see \cite{CorrelationPaper} or \cite{quadrini} and references therein). For Germany and the USA in 2021, Figure \ref{figure: WID} shows the distribution of net personal wealth per adult, which is defined as the total value of non-financial and financial assets (housing, land, deposits, bonds, equities, etc.) held by individuals, minus their debts. Figure \ref{figure: WID} confirms the two-tailed structure since the richest seem to follow a power-law distribution. Least square fit estimates a Pareto exponent of approximately 1.44 for Germany 2021 with similar values for 2011 and the USA. This exponent is similar for other countries as presented in \cite{vermeulen} and widely stable in time \cite{benhabib}, as already predicted by V. Pareto himself. For comparison, \cite{chatterjee} also finds a Pareto-tail in income distribution with exponent varying between 2.42 and 3.96 in Germany since 1990, underling the mentioned gap between income and wealth distribution.

\begin{figure}
  \centering
  \includegraphics[width=0.7\linewidth]{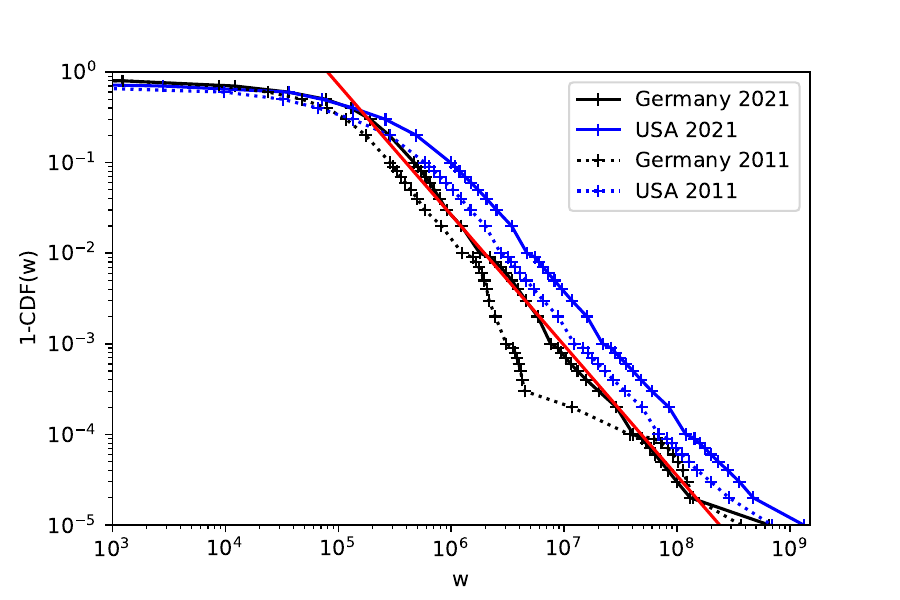}
  \caption{$1-CDF$ (Cumulative Distribution Function) of net personal wealth (purchasing power parity, equal split adults) in Germany and the USA in 2011 and 2021 in Euro resp. US-Dollar according to \cite{wid}. Least square fit (red line) estimates a Pareto exponent of 1.44 for Germany 2021. }
  \label{figure: WID}
\end{figure}

Simple models assuming independent return rates (like the Black-Scholes model) can only account for exponential-tailed wealth distributions, but numerous more complex models have been proposed to find the determinants of the power-law tail, some of which can be found in \cite{forbes, gabaix, benhabib, chatterjee, bouchaud, wold, kohlrausch, benhabib2, cardoso, liu, yakovenko, boghosiane, monaco} and references therein. The early model proposed in \cite{simon}, which was recently taken up in \cite{hu}, already contains the idea of combining independent and reinforced elements. An essential reason for the power-law tail has been found in so-called increasing returns, i.e. the return rates on capital depend on the amount of capital a person or household owns. More casually spoken: the richer you are, the faster your wealth grows. Detailed theoretical and empirical information on this phenomenon can be found e.g. in \cite{Arthur, forbes, fagereng, bach, ederer}. Diverse explanations for increasing returns can be conceived, like lower risk aversion of the rich due to higher risk-bearing potential. Empirical studies in \cite{fagereng} confirm this idea, but on the other hand increasing returns can even be found within similar asset groups. Hence, risk is not the exclusive driver of increasing returns and other factors like higher skills, political influence, informational advantages, decreasing costs of debt, tax-evasion or lower transaction costs are relevant, too. Moreover, some asset classes are barely available for ordinary people, like private equity fonds, art or other value increasing luxury goods.

An established model for growth processes subjected to increasing returns is the generalized Pólya urn model, as introduced by \cite{Hill} and motivated in this context by W. Brian Arthur \cite{Arthur, Arthur2}. A comprehensive survey of the properties of this model is provided by the authors of this paper in \cite{wir} and Appendix \ref{sec: r=0} summarizes the most important features. The main idea is that wealth is added step by step to the economy, where the probability of an agent to receive the next unit of wealth depends on the current wealth of the agent. \cite{Vallejos} fits the classical generalized Pólya urn to American data. Indeed, this model creates Pareto-tailed wealth distributions (see \cite{Oliveira, Zhu} and Appendix \ref{sec: r=0}), but a major drawback of this model is the occurrence of strong monopoly, i.e. from some point on only one (random) agent wins in all following steps (see Appendix \ref{sec: r=0}). The aim of this paper is to extend the generalized Pólya urn model, such that the empirical wealth distribution from Figure \ref{figure: WID} emerges as a stable long-term distribution under the dynamics of our new model.

First, we pick up the main idea of Pólyas urn model and distribute the wealth created in an economy step by step among a fixed number $A\in\{2, 3, \ldots\}$ of agents. But in this paper, we distinguish two different mechanisms of assigning an abstract unit of additional wealth to an agent. Assume that a company generates a unit of additional wealth, which corresponds to the gross yield. A certain share $r\in[0, 1]$ (the so-called "labor share") of the gross yield is payed to the employees via wages. The remaining $1-r$ units (the "profit share") are assigned to the shareholders, either by paying dividends or by increasing the fundamental value of the company. For simplicity, we assume wages to be fixed in time and independent of wealth in our model, i.e. this share of the abstract wealth unit is distributed among all agents proportionate to some fixed vector. The other part of the added wealth unit is distributed among the agents via capital returns and does hence depend on the current wealth of the agents. Capital returns will be modeled as a generalized Pólya urn as in \cite{wir}, which includes the phenomenon of increasing returns. Hence, the remaining $(1-r)$-share of the wealth unit is fully assigned to one randomly chosen agent. The share $r$ can be understood as an adjusted labour share in reality (see Subsection \ref{subsec: r}), which is a measure for the importance of capital for the accumulation of wealth. In more casual words: $r$ regulates in how far it is possible to become rich through hard work. As explained in Subsection \ref{subsec: r}, it is justifiable to assume that the labor share is constant in time. Since the distribution of wages is an exogenous parameter in our model, we will particularly focus on providing an explanation for the discrepancy between wage distribution and wealth distribution as a consequence of increasing returns.

In Section \ref{sec: model}, we will formally introduce the model and present a few rigorous results concerning the long-time behaviour using the method of stochastic approximation (see e.g. \cite{Hasminskii, Pemantle}). In Section \ref{sec: A=2}, we discuss the cases of $A=2$ and $A=3$ agents, in order to gain a visual understanding of the different regimes of the process.  In Section \ref{sec: simulation}, we fit the model parameters to available data and simulate the process for different initial configurations. We compare the simulated wealth distribution to the data from Figure \ref{figure: WID} and take a look at some other inequality indicators in order to reveal advantages and disadvantages of the proposed model. Moreover, we formulate predictions for the future based on our model. In Section \ref{sec: alpha}, we use our model to discuss if different investment skills provide an alternative explanation for the gap between wealth and wage distribution. Finally, in Section \ref{sec: discussion}, we bring together our numerical and theoretical findings and discuss the effect of the recent increase of interest rates on the future of inequality within our model.

\section{The model and some rigorous results}\label{sec: model}

In this section, we formally introduce the model. Let $A\in\{2, 3, \ldots\}$ be the number of agents in our economy and $r\in[0, 1]$ denotes the share of the added wealth, which is distributed proportionate to a deterministic vector $\gamma=(\gamma_1, \ldots, \gamma_A)\in\Delta_{A-1}\coloneqq\{(x_1,\ldots, x_A)\in[0, 1]^A\colon \sum_{i=1}^Ax_i=1\}$, representing the effect of wages on the accumulation of wealth. We fix a feedback function $F_i\colon(0, \infty)\to(0, \infty)$ for each agent $i\in[A]\coloneqq\{1,\ldots, A\}$. Now, we are prepared to define a homogeneous Markov process $X(n)=\left(X_1(n), \ldots, X_A(n)\right)$, $n\in\N_0$ in the state space $[1, \infty)^A$ with initial condition $X(0)\in[1, \infty)^A$. Denote by $N\coloneqq X_1(0)+\ldots+ X_A(0)$ the total wealth at time $n=0$. Then the transition probabilities are given by
\begin{equation}\label{eq: transpr}
    \P\left(X(n+1)-X(n)=v^{(i)}\,\big|\,X(n)\right)=p_i\left(n, \frac{X(n)}{N+n}\right)\coloneqq \frac{F_i(X_i(n))}{F_1(X_1(n))+\ldots+F_A(X_A(n))}
\end{equation}
for $i\in[A]$, where $v^{(i)}\coloneqq (1-r)e^{(i)}+r\gamma$ with $e^{(i)}\coloneqq(\delta_{i, j})_{j\in[A]}$. The corresponding process of wealth shares is defined as 
$$\chi(n)\coloneqq\left(\chi_1(n),\ldots,\chi_A(n)\right)\coloneqq \frac{1}{N+n}X(n)\in\Delta_{A-1},\quad n\in\N_0,.$$
 For $r=0$, this process coincides with the generalized Pólya urn studied in \cite{wir}, whereas for $r=1$ the process $\chi(n)$ is deterministic. Moreover, define the vector field
\begin{equation}\label{eq: G}
    G(n, x)\coloneqq\E\left[X(n+1)-X(n)\,\big|\,X(n)=\lfloor (N+n)x\rfloor\right]-x=(1-r)p(n, x)+r\gamma-x
\end{equation}
of centered expected increments for $x\in\Delta_{A-1}, n\in\N$ and $p(n, x)\coloneqq(p_1(n, x),\ldots, p_A(n, x))$. In particular, the field $G$ represents the expected increment of shares up to scaling, i.e.
\begin{equation}\label{eq: G2}
\E\left[\chi(n+1)-\chi(n)\,|\,\chi(n)=x\right]=\frac{G(n,x)}{N+n+1}\,.
\end{equation}

Note that the time-scale of our model is non-linear, i.e. one step of the process does not correspond to a fixed period of time in reality. When $\mu=\mu(t)$ is the annual growth rate of our economy (given as an exogenous parameter), then we could instead consider the time-changed process $t\mapsto X\left(\lfloor \left((1+\mu)^t-1\right)N\rfloor\right)$, where $t$ is time measured in years. This will be discussed in detail in Section \ref{sec: simulation}, in the following the explicit time scale is not important.

A rigorous approach to the long time behaviour of this process is provided by the method of stochastic approximation, see e.g. \cite{Pemantle, Hasminskii} and references therein. For that, we consider Doob's decomposition 
\begin{equation}\label{eq: Doob}
    \chi(n)=\chi(0)+H(n)+M(n)\,,
\end{equation}
where
\begin{equation*}
H(n)\coloneqq\sum_{k=0}^{n-1}\frac{G(k, \chi(k))}{N+k+1}\quad\text{and}\quad M(n)\coloneqq\sum_{k=0}^{n-1}\frac{1}{N+k+1}\xi(k),
\end{equation*}
with $\xi(n)\coloneqq X(n+1)-X(n)-G(n, \chi(n))-\chi(n)$. The process $H(n)$ is predictable with respect to the filtration $\left(\mathcal{F}_n\right)_n$ generated by the process $\chi(n)$. Moreover, the $\xi(n)$ are centered, bounded and uncorrelated since
$$\E\left[\xi_i(n)\xi_j(m)\right]=\E\left[\xi_i(n)\E\left[\xi_j(m)\,\big|\,\mathcal{F}_n\right]\right]=0\quad\text{for }m>n,\, i, j\in[A]\,.$$
Hence, $M(n)$ is a martingale, which is bounded in $L^2$ and consequently almost-surely convergent for $n\to\infty$.

For simplicity, we will mainly consider homogeneous feedback functions in the following, i.e. $F_i(k)=\alpha_ik^\beta$ for some $\alpha=(\alpha_1,\ldots, \alpha_A)\in(0, \infty)^A$ and $\beta\in\R$. This kind of feedback is particularly simple since the transition probabilities $p(n, x)$ do not depend on $n$ such that we can establish the notation
\begin{equation}\label{eq: gg}
p(x)\coloneqq p(n, x)\quad\text{and}\quad G(x)\coloneqq G(n, x)
\end{equation}
for the homogeneous case. Exploiting the convergence of the martingale $M(n)$, one can show that the long time limits of this process are given by the zeros $x$ of the vector field $G$, i.e. $G(x)=0$, which we will refer to as \textbf{fixed points} of the dynamics.

\begin{theorem}\label{thm: intax}
    For all $i\in[A]$ let $F_i(k)=\alpha_i k^\beta$ for $\alpha_i>0,\,\beta\in\R$, $r\in[0, 1]$. Then $\chi(n)\to\chi(\infty)$ converges almost surely to a stable fixed point of $G$ for $n\to\infty$.
\end{theorem}

\begin{proof}
    The proof follows similar stochastic approximation arguments like in \cite{Arthur2, Benaim, Pemantle, Hasminskii}.  Define the set $S\subset\Delta_{A-1}$ of fixed points of $G$. Similar to \cite{Benaim}, a Lyapunov function for $G$ is given by
    \begin{equation}\label{eq: lyapunov}
        L(x)\coloneqq-(1-r)\log\left(\sum_{i=1}^A\alpha_i x_i^\beta\right)-r\sum_{i=1}^A\gamma_i\log x_i+\sum_{i=1}^A x_i\quad \text{for }x=(x_1,\ldots, x_A)\in\Delta_{A-1}^o
    \end{equation}
    as $\frac{d}{dx_i}L(x)=-\frac{1}{x_i}G_i(x)$ and consequently $\langle\nabla L(x),\,G(x)\rangle=-\sum_{i=1}^A x_i\left(\frac{d}{dx_i}L(x)\right)^2\le0$ with equality if and only if $x\in S$. Now, we observe that $L(\chi(n))$ eventually becomes a supermartingale:
    \begin{align*}
        \E\left[L(\chi(n+1)-L(\chi(n))\,\big|\,\chi(n)\right]&=\E\left[\nabla L(\chi(n))\left((\chi(n+1)-\chi(n)\right)+O(1/n)\,\big|\,\chi(n)\right]\\
        &=\nabla L(\chi(n))\E\left[\chi(n+1)-\chi(n)\,\big|\,\chi(n)\right]+O(1/n)\\
        &=\langle \nabla L(\chi(n)),\,G(\chi(n))\rangle+O(1/n)
    \end{align*}
     Since $L$ is bounded from below, we get almost sure convergence of $L(\chi(n))$ from the martingale convergence theorem. Take an open $\epsilon$-neighborhood $U_\epsilon\subset\Delta_{A-1}$ of $S$. Then there is $\delta(\epsilon)>0$ such that
    $$\E\left[L(\chi(n+1))-L(\chi(n))\,\big|\,\chi(n)=x\right]<-\delta(\epsilon)\quad\text{for all }x\in\Delta_{A-1}\setminus U_\epsilon\,,$$
    if $n$ is large enough. Hence, the limit point needs to be in any $U_\epsilon$ and consequently in $S$.
    
    The non-convergence to unstable fixed points is technically more demanding and follows basically from arguments like in \cite[Lemma 5.2.]{Arthur2}, \cite[Theorem 2.9]{Pemantle} or \cite[Chapter 5]{Hasminskii}. Note that the stable fixed points of $G$ are just the strict local minima of $L$. Maxima and saddle points of $L$, i.e. unstable fixed points of $G$, are not attained as limit points of $L(\chi(n))$ due to noise of order $\frac1n$.
\end{proof}

Like in \cite{Arthur2}, it is possible to extend Theorem \ref{thm: intax} to inhomogeneous feedback functions, provided that the field $G(k, x)$ converges for $k\to\infty$ sufficiently fast. For our applied purposes, these inhomogeneous feedback functions do not grant any enriching insights, such that we will neglect them.

Also note that the equation $G(x)=0$ does in general consist of $A-1$ independent equations for $A-1$ variables, since the $A$-th equation is redundant due to
\[
\sum_{i=1}^AG_i(x)=0\quad\mbox{and}\quad\sum_{i=1}^Ax_i=1\quad\mbox{since }x\in\Delta_{A-1} \ .
\]
Hence, heuristically, the zero-set of $G$ can be considered to be discrete. 

 An interesting observation in the situation of Theorem \ref{thm: intax} is that the limiting share of all agents $i\in[A]$ with positive wage $\gamma_i>0$ is bigger than $r\gamma_i$, i.e. $\chi(\infty)>r\gamma_i$ almost surely for all $i\in[A]$ since $p(n, x)>0$ for all $x\in\Delta_{A-1}^o$. The inequality is strict since all agents do not only receive their wage, but also capital returns on their savings.

 As mentioned before, our process is deterministic for $r=1$. The following Proposition states that the process reveals a deterministic long-time behaviour even for large enough $r<1$. In that case, agents with zero wage will have vanishing shares on the long run.

 \begin{proposition}\label{prop:det}
      For all $i\in[A]$ let $F_i(k)=\alpha_i k^\beta$ for $\alpha_i>0,\,\beta\ge1$. Then there is a critical labor share $r_c<1$, such that for all $r\ge r_c$ the limit $\chi(\infty)\coloneqq\lim_{n\to\infty}\chi(n)$ is deterministic. If $r\ge r_c$, then any agent $i\in[A]$ with $\gamma_i=0$ necessarily fulfills $\chi_i(\infty)=0$.
 \end{proposition}

 \begin{proof}
     First, Theorem \ref{thm: intax} implies the existence of a fixed point of $G$.  For $r\in(0, 1)$ take $x_r, y_r\in\Delta_{A-1}^o$ satisfying $G(x_r)=G(y_r)=0$. Define $G_0(x)\coloneqq p(x)-x$ for $x\in\Delta_{A-1}$ and $p(x)\coloneqq p(k, x)$. Note that the field $G_0$ does not depend on $r$. Then:
     \begin{align*}
         G_0(x_r)-G_0(y_r)&=\frac{1}{1-r}\left((1-r)p(x_r)+r\gamma-x_r-(1-r)p(y_r)-r\gamma+y_r+rx_r-ry_r\right)\\
         &=\frac{1}{1-r}\left(G(x_r)-G(y_r)+rx_r-ry_r\right)=\frac{r}{1-r}(x_r-y_r)
     \end{align*}
     As $\beta\ge1$, $G_0$ is Lipschitz-continuous with a Lipschitz-constant $L=L(\alpha, \beta)<\infty$. Hence, we have $x_r=y_r$, when $\frac{r}{1-r}>L$, and as a consequence $G$ has only one unique fixed point for $r\ge r_c :=\frac{L}{1+L} \in (0,1)$.

     Now, let $r\ge r_c$ and assume (w.l.o.g.) $\gamma_1=0$.  If $y_0\in\Delta_{A-2}$ is a zero of the restricted field $\tilde G(y)=G(0, y),\,y\in\Delta_{A-2}$, which corresponds to a system with $A-1$ agents, then $(0, y_0)$ is a zero of $G$. Uniqueness of the fixed point for $r\ge r_c$ and the existence of such $y_0$ imply $\chi_1(\infty)=0$.
 \end{proof}

This proof also implies that there is no further (unstable) fixed point of $G$ for $r\ge\frac{L}{1+L}$. Note that Proposition \ref{prop:det} may hold with $r_c <\frac{L}{L+1}$, our argument provides only an upper bound for $r_c$.

For $r=0$, our process equals the generalised Pólya urn studied in \cite{wir} and references therein. If in addition $\beta>1$, then the process reveals strong monopoly with probability one, i.e. at some point one agent wins all following steps. Of course, this cannot happen for $r>0$ since all agents get at least their wage and receive capital returns on their wage on top. Nevertheless, we suggest to call agents  \textbf{winner},  if their wealth share exceeds their income share, i.e. $\chi_i(\infty)>\gamma_i$. Otherwise we call them loser. In that sense, we can still identify a unique random winner for small enough $r>0$.

\begin{proposition}
    Let $F_i(k)=k^\beta$ for all $i\in[A]$ with $\beta>1$. Then there is $r_c'>0$ such that for all $r<r_c'$
    $$\P\left(\exists !\,i\in[A]\colon\chi_i(\infty)>\gamma_i\right)=1$$
    and for all $i\in[A]$
    $$\P\left(\chi_i(\infty)>\gamma_i\right)>0\,.$$
\end{proposition}

\begin{proof}
    Let $i\in[A]$. Denote by $DG(x)\coloneqq \left(\frac{\partial G(x)}{\partial x_i\partial x_j}\right)_{i, j}$ the differential matrix of $G$ in $x\in\Delta_{A-1}$. For $r=0$, a simple computation shows $\nabla G_j(e^{(i)})=(-\delta_{l, j})_{l=1,\ldots,A}$ for all $j\in[A]$. Hence, $DG(e^{(i)})$ is negative definite and invertable. Then we get from the implicit function theorem that there is $\epsilon>0,\,r_c'>0$ such that for all $r<r_c'$ there is exactly one zero of $G$ in the $\epsilon$-neighborhood of $e^{(i)}$. Denote this fixed point by $x^{(i)}(r)$. Obviously, for all agents $j\ne i$ with $\gamma_j=0$ we must have $x_j^{(i)}(r)=0$ due to the uniqueness of the fixed point. Hence, assume without loss of generality that $\gamma_j>0$ for all $j\ne i$ and suppose $\epsilon\le\min\{\gamma_j\colon j\ne i\}$. Consequently, $x_j^{(i)}(r)\le\gamma_j$ for all $j\ne i$ and $r<r_c'$. 

    It remains to show that for $r<r_c'$ there are no other stable fixed points of $G$. Therefor, consider the "$r{=}0$"-field $G_0(x)=p(x)-x$. We know from \cite{wir} that all zeros of $G_0$ have the form $x^{(S)}=\left(\frac{1}{\#S}\mathds{1}_{S} (i)\right)_{i\in[A]}$ for a non-empty subset $S\subset[A]$. Since $G$ is a continuous perturbation of $G_0$, we know that for small enough $r$ all zeros of $G$ are located in an $\epsilon$-neighborhood of these points $x^{(S)}$. Moreover, $x^{(S)}$ is unstable for $\#S>1$ such that $DG_0(x^{(S)})$ has at least one positive eigenvalue. Since eigenvalues of $DG(x)$ do continuously depend on $x$ and $r$, there is still at least one positive eigenvalue of $DG(x)$ for any zero $x$ of $G$ that is located in an $\epsilon$-neighborhood of $x^{(S)}$ with $r$ small enough. Hence, all stable fixed points are close to an $x^{(S)}$ with $\#S=1$ if $r$ is small. The stability of these fixed points can be shown similarly using negative definiteness of $DG_0(e^{(i)})$.
\end{proof}

An interesting implication of the construction of our stable fixed points is the following: For $r<r_c'$, there is exactly one fixed point close to each corner of $\Delta_{A-1}$. Hence, $\chi(\infty)$ is fully determined by picking the winner, i.e. there is no random hierarchy between the losers.  In the next section, we will see that the fixed points disappear one by another, when $r$ is increased, until finally only one fixed point remains for $r\ge r_c$.

Since it will be of particular interest in Section \ref{sec: alpha}, let us now discuss the linear case $F_i(k)=\alpha_i k$ with skill parameter $\alpha_i>0$, which corresponds to wealth independent return rates. Obviously, $\gamma$ is the unique stable fixed point of $G$ when $\alpha_1=\ldots=\alpha_A$ and $r>0$ and hence $\chi(n)$ converges to $\gamma$ almost surely. In case of the standard Pólya urn (i.e. $\alpha_i=1$ and $r=0$), $\chi(n)$ converges almost surely towards a random point. For unequal $\alpha_i$ and $r=0$, the process reveals a deterministic weak monopoly, i.e. $\chi(n)$ converges to $e^{(i)}$, where $i$ is the agent with the largest $\alpha$ (see Appendix \ref{sec: r=0}). Indeed, there is a deterministic limit for all choices of $r>0$ and $\alpha_i>0$. This does also hold in the sublinear case, where $\chi(\infty)$ is deterministic even for $r=0$.

\begin{proposition}\label{prop: alpha}
    Let $r>0$ and $F_i(k)=\alpha_i k^\beta$ for $\alpha_i>0$ and $\beta\le1$. Then $\chi(n)$ converges almost surely to a deterministic point $\chi (\infty )$ for $n\to\infty$, i.e. $r_c=0$.
\end{proposition}

\begin{proof}
    Using the argument from the proof of Theorem \ref{thm: intax}, we have to show that the Lyapunov function $L$ defined in (\ref{eq: lyapunov}) has a unique minimum. For that, we prove that $L$ is strictly convex. Direct calculation yields that the Hessian of $L$ is of the form
    $$\left(\frac{\partial L(x)}{\partial x_i\partial x_j}\right)_{i, j}=c(x)\cdot v\cdot v^T+A(x),\quad x=(x_1,\ldots, x_A)\in\Delta_{A-1}^o\,,$$
    where $c(x)\coloneqq(1-r)\beta^2\left(\sum_{i=1}^A\alpha_i x_i^\beta\right)^{-2}\ge0$ and $v=\left(\alpha_i x_i^{\beta-1}\right)_{i\in[A]}\in (0, \infty)^A$ and $A(x)=\left(A_{i, j}(x)\right)_{i, j\in[A]}$ is a diagonal matrix with
    $$A_{i, i}(x)=r\gamma_i x_i^{-2}+(1-r)\beta(1-\beta)\alpha_i^2x_i^{\beta-2}\left(\sum_{j=1}^A\alpha_j x_j^\beta\right)^{-1}\ge 0$$
    Assume $r<1$ as $r=1$ is trivial. Since $v\cdot v^T$ is non-negative definite, the Hessian of $L$ is positive definite if either $\gamma_i>0$ for all $i\in[A]$ or $\beta<1$. But if $\beta=1$, we can w.l.o.g. assume $\gamma_i>0$ due to Lemma \ref{lemma: beta=1, gamma=0}. Hence, $L$ is strictly convex.
\end{proof}

\begin{lemma}\label{lemma: beta=1, gamma=0}
    Let $r>0$ and $F_i(k)=\alpha_i k$ for $\alpha_i>0$. Then $\chi_i(\infty)=0$ for any agent $i\in[A]$ with $\gamma_i=0$.
\end{lemma}

\begin{proof}
    Due to the linearity, if suffices to consider a system with $A=2$ and $\gamma_1=1$, since this process is equivalent to the group-process $\left(\sum_{i\in[A]\atop \gamma_i>0}\chi_i(n),\,\sum_{i\in[A]\atop \gamma_i=0}\chi_i(n)\right)_n$. But then a simple calculation shows that
    $$G(x)=0\quad\Leftrightarrow\quad (1-r)\frac{\alpha_1 x_1}{\alpha_1 x_1+\alpha_2(1-x_1)}+r=x_1$$
    has only the solution $x_1=1$.
\end{proof}

\begin{figure}
  \centering
  \subfloat{\includegraphics[width=0.5\linewidth]{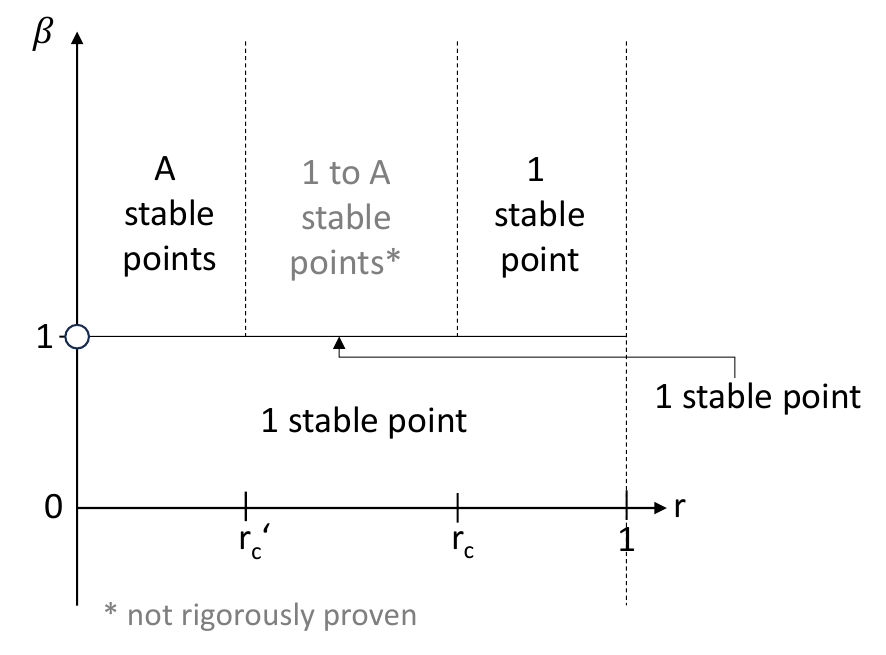}}%
  \caption{Qualitative illustration of the number of stable fixed points of $G$ for homogeneous feedback and different $r$ and $\beta$. $\circ$ marks the classical Pólya urn, which exhibits either deterministic (weak) monopoly or a Dirichlet distributed limit. }
  \label{figure: rb_image}
\end{figure}

Let us finally sum up the homogeneous case $F_i(k)=\alpha_i k^\beta$: The limit $\lim_{n\to\infty}\chi(n)$ does exist and has a discrete distribution, but both deterministic and random limits are possible depending on $r$ and $\beta$. Remarkably, the set of possible limit points does not depend on the initial configuration $X(0)$, but the probability that a specific limit point is attained might depend on $X(0)$. This means for $r<r_c$ that agents with high wage and low initial wealth may never fully catch up their initial disadvantage, whereas they will do so for $r\ge r_c$. For $\beta>1$ and $r<r_c$, the first steps of the process decide which limit point is attained, because the process behaves almost deterministic for large market sizes according to the law of large numbers presented in Appendix \ref{sec: llnTax}. For small enough $r$, we can still identify unique winners and any agent can be the winner. For constant or decreasing return rates $\beta\le1$, $\lim_{n\to\infty}\chi(n)$ is deterministic for all $r>0$ and hence $r_c=0$, whereas $r_c>0$ in the increasing return case $\beta>1$. Figure \ref{figure: rb_image} finally illustrates these findings.

 \section{The two and three agent case}\label{sec: A=2}

In order to gain a visual understanding of the long-time behaviour of this process, we will discuss the homogeneous case with $A=2$ in detail. So, let $F_1(k)=k^\beta$, $F_2(k)=\alpha k^\beta$ for $\beta\in\R$ and $\alpha\ge1$. For simplicity, we establish the notation $G(x)=G_1(x, 1-x)$, $\gamma=\gamma_1$ and $p(x)=p_1(x, 1-x)$ for $x\in[0, 1]$, i.e. $x$ represents the share of agent $1$ and agent 2 has share $1-x$. Then
\begin{equation*}
    G(x)=(1-r)(p(x)-x)+r(\gamma -x)=0\quad\Leftrightarrow\quad 
    p(x)-x=\frac{r}{1-r}(x-\gamma)
\end{equation*}
and the stable fixed points of $G$ are the downcrossings of the "$r=0$"-field
\begin{equation}\label{eq: fieldG0}
    x\mapsto G_0(\alpha, \beta; x)\coloneqq p(x)-x
\end{equation}
with the line 
\begin{equation}\label{eq: lineg}
    x\mapsto g(r, \gamma; x)\coloneqq\frac{r}{1-r}\left(x-\gamma\right)\,.
\end{equation}
These downcrossings constitute the possible long-time limits of the process $\chi(n)$ according to Theorem \ref{thm: intax}. The upcrossings are unstable fixed points and are not attained as long-time limits.

\begin{figure}
  \centering
  \subfloat{\includegraphics[width=0.5\linewidth]{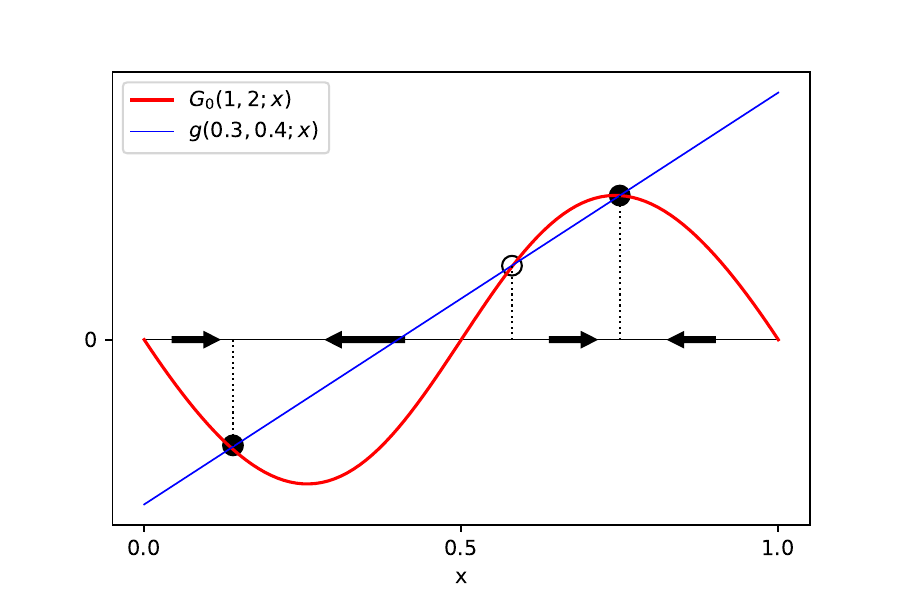}}%
  \subfloat{\includegraphics[width=0.5\linewidth]{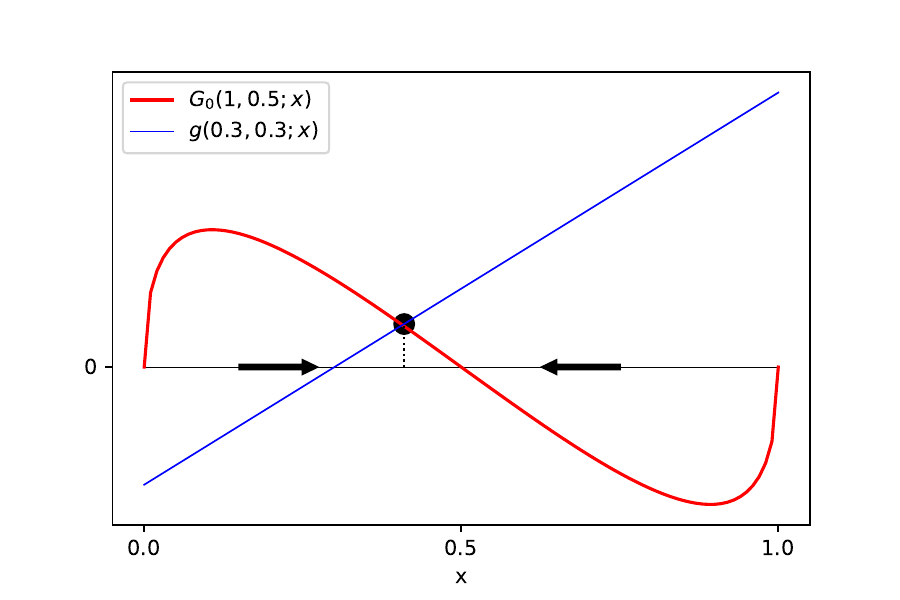}}%
  \caption{The line $g$ (see \ref{eq: lineg}) and the field $G_0$ (see (\ref{eq: fieldG0})) against wealth $x$ of agent 1. $\bullet$ marks stable and $\circ$ unstable fixed points. The arrows indicate the direction of the field $G$. }
  \label{figure: fieldA2}
\end{figure}

\begin{figure}
  \centering
  \subfloat[][]{\includegraphics[width=0.5\linewidth]{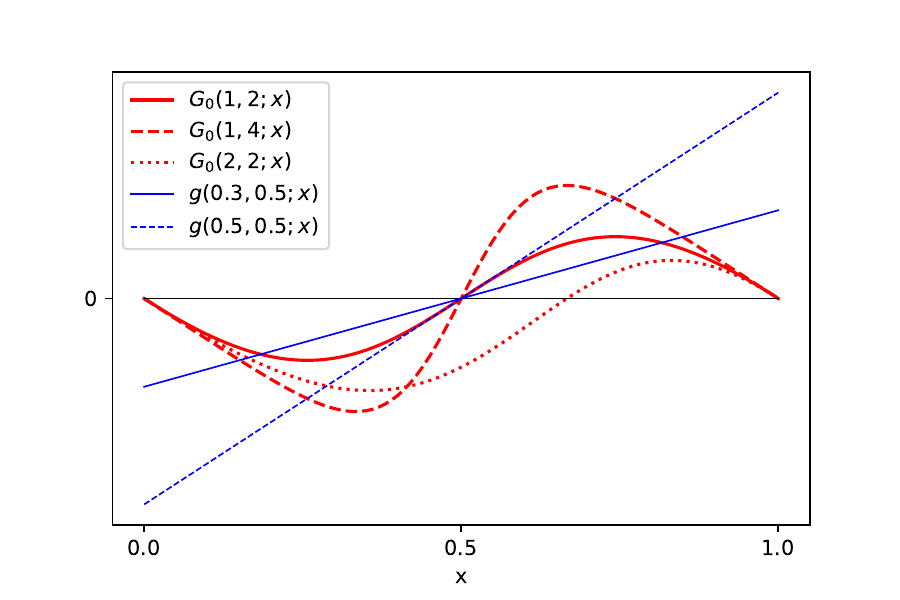}}%
  \subfloat[][]{\includegraphics[width=0.5\linewidth]{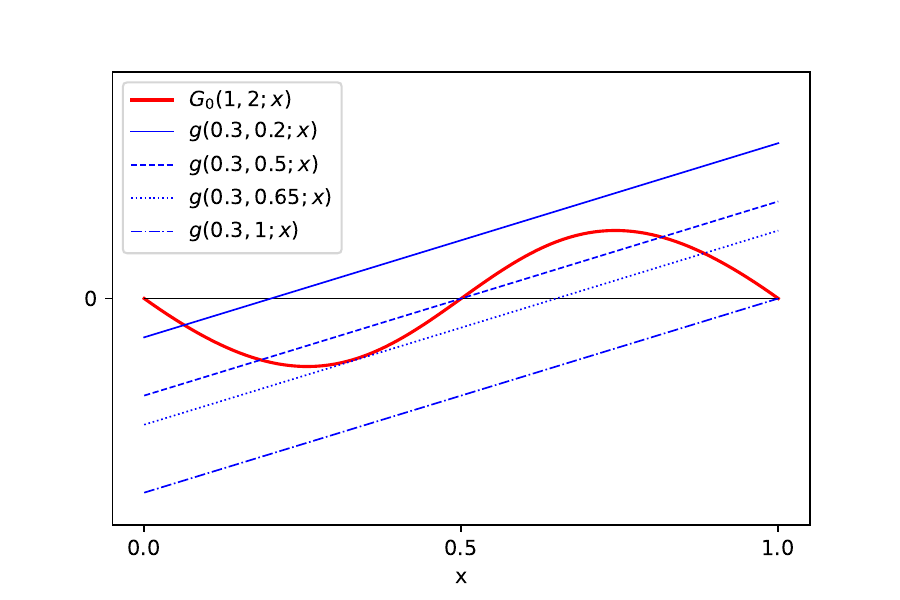}}\\
  \subfloat[][]{\includegraphics[width=0.5\linewidth]{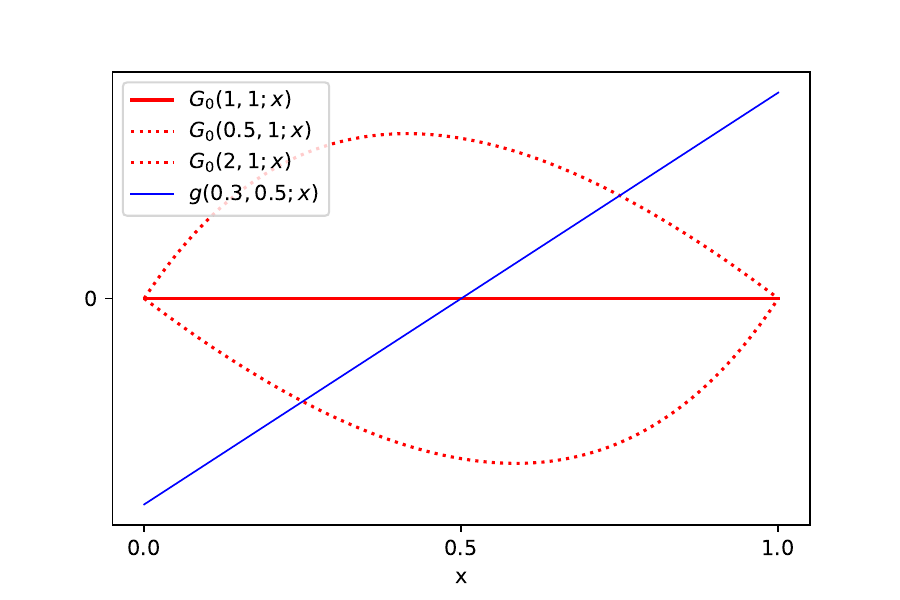}}%
  \subfloat[][]{\includegraphics[width=0.5\linewidth]{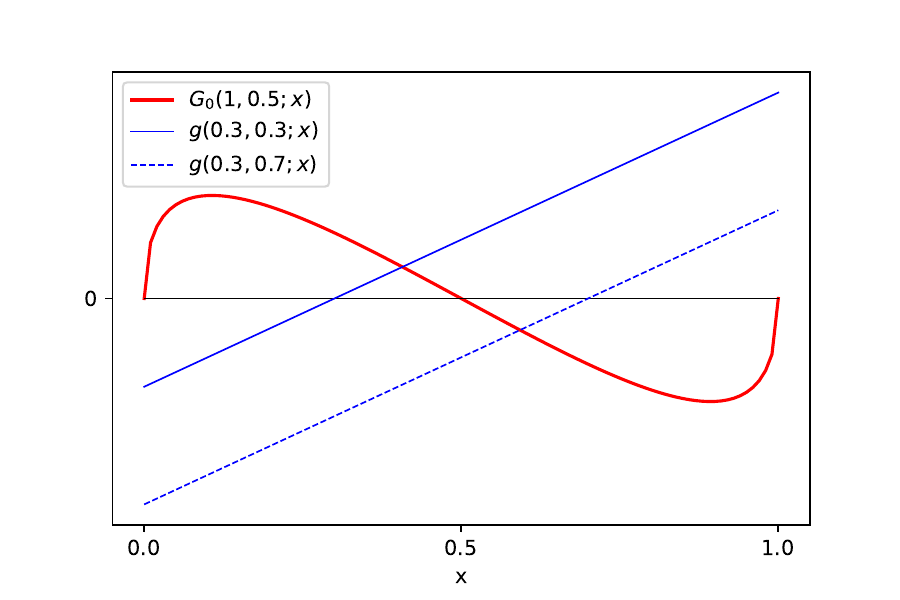}}%
  \caption{The line $g$ (see \ref{eq: lineg}) and the field $G_0$ (see (\ref{eq: fieldG0})) against wealth $x$ of agent 1 for various parameters.}
  \label{figure: intax}
\end{figure}

The situation is qualitatively illustrated in Figure \ref{figure: fieldA2} and \ref{figure: intax}. An increase of the labor share $r$ implies a larger slope of the line $g$, where the slope diverges for $r\to1$. Changes of the relative wage $\gamma$ result in a parallel shift of $g$. The impact of the fitness parameter $\alpha$ and the feedback strength $\beta$ is included in the field $G_0$. Let us now have a closer look on the possible cases.

\begin{enumerate}
    \item Figure \ref{figure: intax} (a) shows the symmetric case $\alpha=1, \gamma=0.5$ with $\beta>1$ for different labor shares $r$. It is apparent that $x=\frac12$ is the only stable fixed point of $G$ if and only if
    $$\frac{d}{dx}(p(x)-x)=\beta-1\le \frac{r}{1-r}\quad\Leftrightarrow\quad r\ge r_c\coloneqq\frac{\beta-1}{\beta}\,.$$
    For $r<r_c$, there are two stable fixed points which are symmetric w.r.t. $\frac12$. For $r\to0$, the two fixed points converge to 0 resp. 1, consistent with the strong monopoly for $r=0$. The critical share $r_c$ is increasing in $\beta$ due to the stronger feedback and converges to 1 for $\beta\to\infty$.
    
    More explicit for $\beta=2$, we have $r_c=\frac12$ and for $r<\frac12$ the two stable fixed points are given by $\frac12\pm \sqrt{1-2r}$. For other choices of $\alpha, \beta$ explicit expressions are lengthy or not known.

    In the asymmetric case $\alpha>1$, where agent 2 is fitter than agent 1, we observe in general a shift of the stable fixed points towards agent 2. Moreover, the critical share $r_c$ is smaller than for $\alpha=1$.

    \item Figure \ref{figure: intax} (b) illustrates the situation with equal fitness $\alpha=1$ of agents,  varying wage distribution $\gamma$ and fixed $r>0$. First, we note that the critical rate $r_c$ is smaller when wages are distributed unequally $\gamma\ne0.5$, i.e. for fixed $r$ we can obtain either random or deterministic limits depending on $\gamma$.

    Second, we observe that for $\beta>1$ the long-time wealth is distributed more unequal than wages. To be more precise, if $\gamma<\frac12$, then $\chi_1(\infty)<\gamma$ and vice versa. The gap between $\gamma$ and $\chi_1(\infty)$ is bigger the smaller $r$ and the larger $\beta$ is.

    Third, for fixed $r<r_c$, there there two choices of $\gamma$, such that \textbf{saddle points} occur (see e.g. line $\gamma'''\approx0.65$). These saddle points are stable from one side (from the left in Figure \ref{figure: intax}), but unstable from the other side. Hence, the process may stick to these points for a long time, but will finally escape towards the only fully stable point due to noise.

    Finally, for $\gamma=1$, weak monopoly of agent 1, i.e. $\chi_1(\infty)=1$, is possible with positive probability. But for $r<r_c$, both weak monopoly of agent 1 and positive shares for both agents are possible, depending on who wins the first steps of the process.

    \item Figure \ref{figure: intax} (c) shows the linear case $\beta=1$, where we have unique fixed points whenever $r>0$, such that $r_c=0$.  For $\alpha=1$, the fixed point is simply $\gamma$. Recall that for $r=0$ the limiting share $\chi_1(\infty)$ has a beta distribution. Changes of $\alpha$ for $r>0$ result to a distortion of the unique fixed point towards the fitter agent.

    \item  Figure \ref{figure: intax} (d) shows the situation for $\beta<1, \alpha=1$, where we still have $r_c=0$. This situation corresponds to decreasing returns in the interpretation  presented in the introduction. Here, wealth is distributed more equally than wages, i.e. we have $\chi_1(\infty)>\gamma$ whenever $\gamma<\frac12$ and $r>0$.
\end{enumerate}

In 2. we already mentioned the occurence of saddle points, which the process may approach and remain there for long time, but will eventually leave. Although this behavior seems to be rare in the two agent case as they do only occur for specific pairs of $\gamma, r$, these points are far more important in larger systems. Hence, we also have a close look on the $A=3$ case in order to deepen our understanding of the long time behavior of our process. Figure \ref{figure: field} shows the field $G$ for asymmetrical wage vector and varying $r$. We return to the original notation of $G$ introduced in Section \ref{sec: model}.

\begin{figure}
  \centering
  \subfloat[][$r=0.3$]{\includegraphics[width=0.4\linewidth]{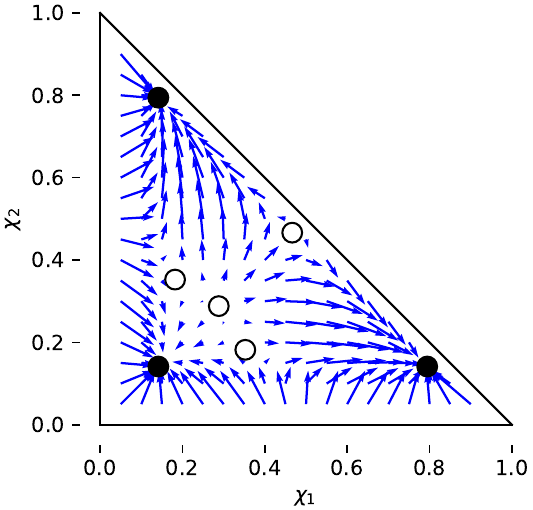}}%
  \subfloat[][$r=0.35$]{\includegraphics[width=0.4\linewidth]{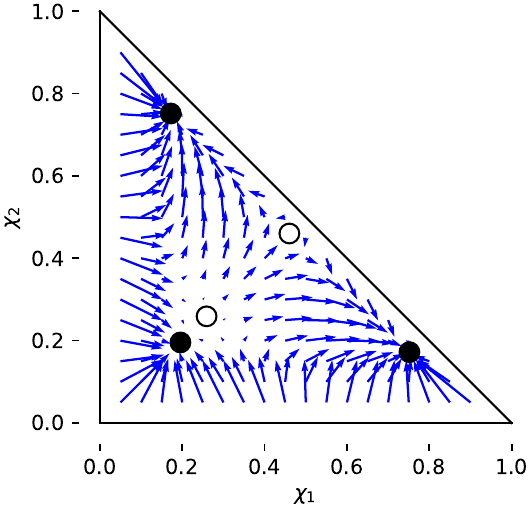}}\\
  \subfloat[][$r=0.4$]{\includegraphics[width=0.4\linewidth]{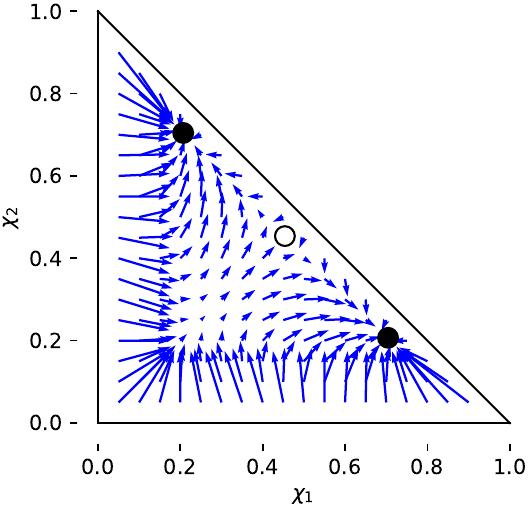}}%
  \subfloat[][$r=0.55$]{\includegraphics[width=0.4\linewidth]{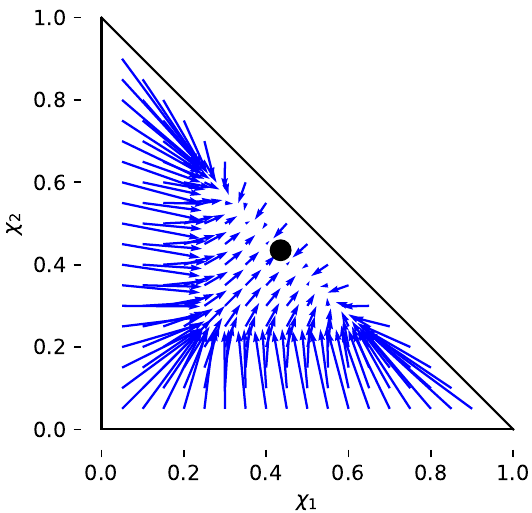}}%
  \caption{The field $G$ (see \eqref{eq: gg}) with $A=3$, $\beta=2, \alpha_1=\alpha_2=\alpha_3=1$ and $\gamma=(0.4, 0.4, 0.2)$ for various  labor shares $r$. Stable fixed points are marked with $\bullet$ and unstable fixed points with $\circ$. Their exact position has been computed with Mathematica.}
  \label{figure: field}
\end{figure}

For small $r$ (Figure \ref{figure: field} (a)), we have 7 fixed points, where the three stable ones are close to the corners of the simplex, i.e. one random agent wins the bulk of the wealth. Casually speaking, we will call these \textbf{monopoly fixed points} in the following. Moreover, there are three saddle points, where basically two agents fairly share the total wealth. For appropriate starting points, the process may first approach these points, but finally converge to one of the stable points. In addition, there is one more fully repelling point in the middle. Note that this situation is similar to the $r=0$ case discussed in \cite{wir}.

When we now slowly increase $r$ to some moderate level ((Figure \ref{figure: field} (b)), we observe that at first the saddle points, where the agent 3 (the one with the lowest wage) is involved, disappear, but the monopoly fixed point of agent 3 still exists. When $r$ is increased even more ((Figure \ref{figure: field} (c)), then the monopoly fixed point of agent 3 also disappears and only the monopoly fixed points of the agents with higher wage remain together with their saddle point. Finally, when $r$ becomes larger than $r_c$ ($\approx0.55$ in the situation of Figure \ref{figure: field}), then the two remaining stable fixed points merge and the process converges to a deterministic point.

Heuristically, we can generalize this visual grasp to larger systems $A>3$ as follows. When $r$ is small, then we have $A$ stable fixed points, which are close to the corners of the simplex. Moreover, for any subset $S\subset\{1,\ldots, A\}, \,\#S>1$, there is a saddle point, where the bulk of the total wealth is shared between the agents in $S$. The saddle points are not attained as long-time limits of the process, but they may dominate the transient behaviour of the system which is relevant in practice. In total, there are 
$2^A-1$ fixed points for small $r$. When $r$ is increased, fixed points shift towards the middle of the simplex and disappear one by another, where the monopoly fixed points of agents with low wage as well as the saddle points referred to them disappear first. $r_c$ is the minimal rate, such that only one of the fixed points survives and $r_c'$ is the labor share, where the first monopoly fixed point disappears. Consequently for "moderate" $r$, the process converges to a random monopoly fixed point, but only agents with large enough wage can be the monopolist. Facing this heuristic, we conjure that there are at most $2^A-1$ fixed points in total and that at most $A$ of them are stable.

\section{Simulations for homogeneous feedback}\label{sec: simulation}

The goal of this section is to find a parametrisation of the model introduced in Section \ref{sec: model}, such that the real distribution of wealth in Germany 2021 is reproduced as well as possible by simulation. We denote by
\begin{equation}\label{eq: CDFger}
    CDF_{ger} :\R\to [0,1]\quad\mbox{the empirical distribution function of wealth, Germany }2021\,
\end{equation}
according to the data from \cite{wid} (shown in Figure \ref{figure: WID}) and compare it to the simulated wealth distribution function $CDF_{sim}$ defined in (\ref{eq: SimulationCDF}). Currently, about 70 million adults are living in Germany, but the data on wealth distribution from \cite{wid} have a much lower resolution. Simulating a system with millions of agents would therefore be computationally very demanding with essentially no verifyable benefit. Therefore we aggregate and take $A=10,000$, such that each agent represents $0.01\%$ of the adults in reality. In particular, the $k$-th richest agent represents the $(k-1)/10,000$ to $k/10,000$ quantile in reality, which are about $7,000$ adults.

According to \cite{wid}, the average net personal wealth per adult in Germany 2021 amounts to 227,567€. One unit of wealth in our model corresponds to 10€ in reality, which is a rather fine resolution. Note that our model is also  scale-invariant in the sense that transition probabilities \ref{eq: transpr} are invariant under a re-scaling of $X(n)$, therefore the choice of wealth units is not critical. We will simulate $n=280,000,000$ steps of our process, such that the average wealth after $n$ steps equals approximately the average wealth in reality. We will see in Figure \ref{figure: Simulations} that the wealth distribution is stable after $n$ steps, so that distributing wealth in smaller units would would not yield any further insight. The wealth distribution after simulating the model \eqref{eq: transpr} for $n$ steps is then given by
\begin{align}\label{eq: SimulationCDF}
    CDF_{sim}\colon \R\to[0, 1],\,w\mapsto\frac{1}{A}\sum_{i=1}^A\mathds{1}_{\{10X_i(n)\le w\}} \ .
\end{align}

We aim to reproduce the $CDF_{ger}$ from generic small initial data and take an initial configuration $X(0)$, such that each agent has only one unit on average. Recall that the set of possible limiting wealth distributions is independent of $X(0)$ as explained in Section \ref{sec: model}, but the probabilities that a certain limit point is attained, does depend on $X(0)$.

In this section, we first consider symmetric and homogeneous feedback, i.e. $F_i(k)=k^\beta$ for some $\beta>1$, and set $\alpha_i=1$ for all agents, i.e. we neglect the effect of personal skills on capital returns. Of course, this is a simplifying assumption as there is probably a positive correlation between investment skills and wages which we will investigate in Section \ref{sec: alpha}. To first approximation, it appears justified to assume that similarly affluent agents invest their money similarly and therefore achieve similar expected return rates. 
This is in particular plausible 
since each agent in our simulation represents 7,000 people in a corresponding wage-class in reality, so we can not account for completely untypical behavior of some individuals anyway. In addition, we can think of people to improve their investment skills the more capital they have for investment, which are thus correlated with wealth (as captured by our parameter $\beta>1$) rather than with wages.
In Subsections \ref{subsec: gamma}, \ref{subsec: r} and \ref{subsec: beta} we fit the parameters $\beta, \gamma$ and $r$ and finally show simulation results in Subsection \ref{subsec: sim}.

\subsection{The wage-vector $\gamma$}\label{subsec: gamma}

In our model, the distribution of wages is an exogenous parameter, which is invariant in time and is represented by the normalized vector $\gamma$. Wages have a significant impact on the wealth of poor agents, whereas the wealth of the rich is mainly determined by capital returns. As this work focuses on modelling the wealth distribution of the rich, i.e. the power law tail mentioned in the introduction, we are content with a rather rough wage-model. We will use data on German wages in 2018 as given in \cite{destatis} (see Figure \ref{figure: r} (a)). Note that only the shape of the wage-distribution is relevant since $\gamma$ is normalized. The wage distribution is derived from data on income tax and contains income from employed and self-employed labor. Capital returns are basically not included as it is not subject to income tax in Germany (there is a flat-rate tax instead). The only exception worth mentioning is rental income, which only poses 2\% of total income and can hence be neglected, too. Let $\tilde\gamma_1,\ldots\tilde\gamma_{A}$ be drawn independently from the distribution shown in Figure \ref{figure: r} (a), where we assume uniform distributions within the intervals. For the agents with wage $>1,000,000$ Euro, we suppose an exponential tail and thus all $\tilde\gamma_i$ are distinct.

\begin{figure}
  \centering
   \subfloat[][]{\includegraphics[width=0.5\linewidth]{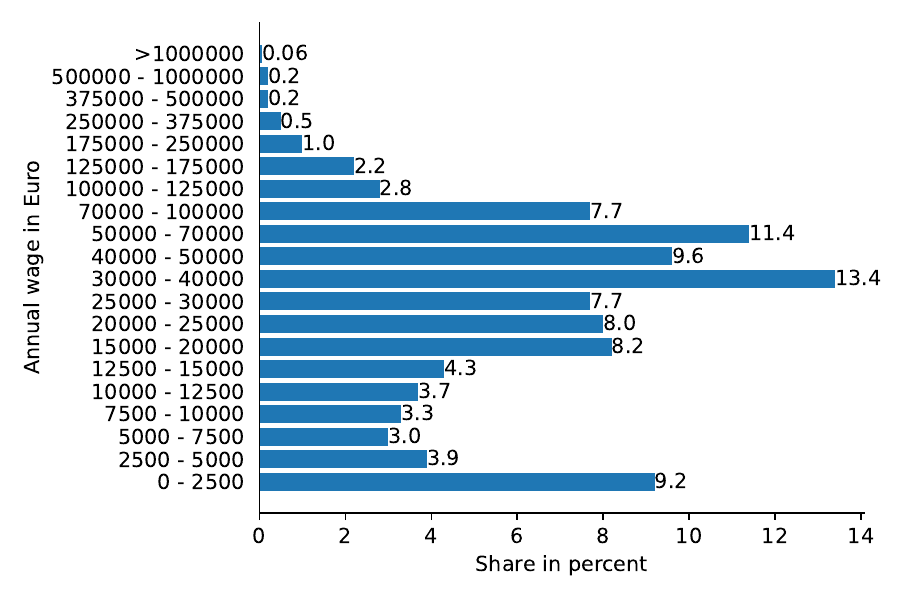}}%
  \subfloat[][]{\includegraphics[width=0.5\linewidth]{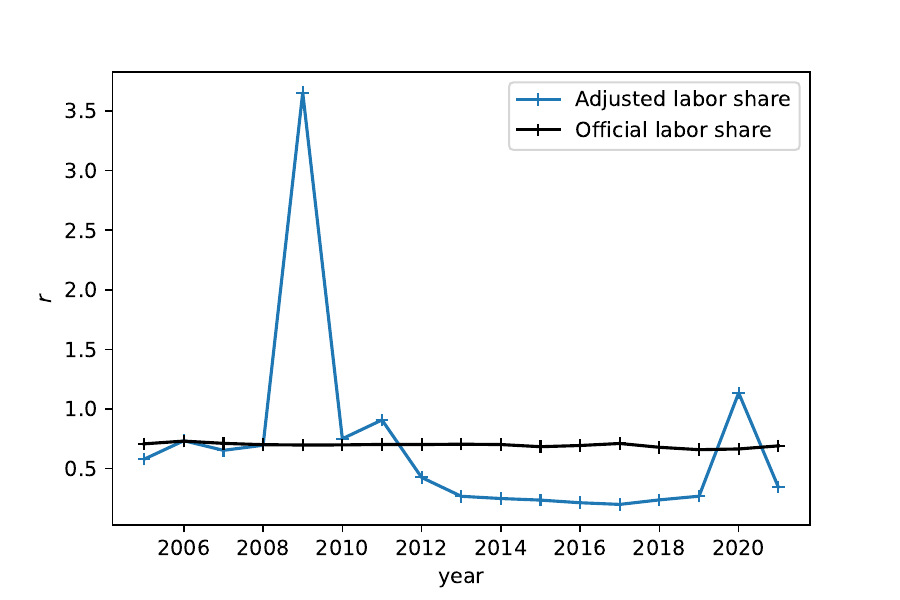}}%
  \caption{(a) shows the distribution of annual net wages in Germany 2018 per taxpayer based on \cite{destatis}. (b) presents empirical values of the labor share r for several years in Germany computed according to formula (\ref{eq: r}) using data from \cite{wid} and \cite{destatis2} compared to the official labor share from \cite{laborShare}.}
  \label{figure: r}
\end{figure}

In order to generate realistic wealth distributions in our model, it is important to distinguish that wages can either be used for consumption or for investment. Hence, we are less interested in the pure wage distribution modelled by $\tilde\gamma$ than in the distribution of savings, which add to the wealth of an agent and generate capital returns. It stands to reason that the agents with lowest wage need essentially all of it for consumption, whereas the highest-payed agents can invest almost their entire wages since even luxury consumption like jewelry and real estate increases value. For simplicity, we assume a linear relationship between the saving rate and the index of ordered wages $\tilde\gamma_{1:A}<\ldots<\tilde\gamma_{A:A}$, such that the agent with rank $i{:}A$ saves a fraction $\frac iA$ of their wage. Hence, we define the normalized sample $\gamma$ as
$$\gamma_i\coloneqq\frac{\tilde\gamma_{i:A}\, i/A}{\sum_{j=1}^A \tilde\gamma_{j:A}\, j/A} =\frac{i\tilde\gamma_{i:A}}{\sum_{j=1}^A j\tilde\gamma_{j:A}}\quad\text{for all }i\in[A]\ .
$$
More detailed information on saving rates depending on income in Germany can be found in e.g. \cite{sparquote}, confirming our linear interpolation as an appropriate approximation. Of course, there is much more refined research on the distribution of income like \cite{chatterjee}, most of which include capital returns in their data and are therefore not suitable for our purpose.

\subsection{The labor share $r$}\label{subsec: r}

In official macroeconomic accounting, the labor share is defined as the part of the national income allocated to wages, which fluctuates in Germany between 64\% and 72\% since reunification 1991 \cite{laborShare}. In our model, however, the parameter $r$ rather represents the part of the wealth increase that is due to savings from wages. Hence, it is not useful to simply set $r\approx0.7$ for several reasons. First, national income does not encompass an increasing value of existing assets like real estate or corporate stocks, which reinforces the significance of capital on the growth of personal wealth. Second, the national income contains consumption, which does not add to wealth. The share used for consumption is presumably higher for wages than for capital returns, which again increases the significance of capital returns for wealth aggregation. Third, the effect of different taxation is not taken into account in the official labor share. 

As a consequence, we estimate the parameter as
\begin{equation}
    r=\frac{\text{average net wage $*$ average savings rate}}{\text{increase of average wealth}}
    \label{eq: r}
\end{equation}
for a fixed period of time. For the increase of average personal wealth, we take data from \cite{wid} again. \cite{destatis2} provides information about average net wages and saving rates.

Figure \ref{figure: r} (b) shows empirical values of $r$ according to formula (\ref{eq: r}) for several years. We observe extreme peaks in 2009 and 2020, which are due to the financial resp.\ the Covid crisis, where the increase of wealth was small, whereas wages are less sensitive to such events. Before 2020, the saving rate fluctuated slightly around 10\%.  Between 2013 and 2019, the empirical r values are stable between 20\% and 27\% percent. This low level is due to strong increases in value of real estate and stocks, caused by zero interest politics. Before the financial crisis, our adjusted labor share widely coincided with the official share. In the following, we will mostly use $r=0.3$ but also consider higher values and show in detail how they affect our results in Section \ref{sec: discussion}.

\subsection{The reinforcement parameter $\beta$}\label{subsec: beta}

After fixing the parameters $\gamma$ and $r$, we finally have to find an appropriate choice for $\beta$. The parameter $\beta$ regulates the reinforcement mechanism of increasing returns, i.e. $\beta=1$ corresponds to constant expected return rates and $\beta>1$  corresponds to increasing returns. Hence, reinforcement for $\beta >1$ determines the deviation of the wealth distribution from the wage distribution. We will estimate $\beta$ directly from the shape of the desired wealth distribution shown in Figure \ref{figure: WID} by adjusting $\beta$ such that $CDF_{ger}$ is fairly stable under the dynamics \eqref{eq: transpr}, since the empirical wealth distribution can be considered as stable in time up to scaling. From (\ref{eq: G}), it is easy to see that for any fixed $\beta$ and $x$, there is a unique $r$ minimizing $\|G(x)\|$. 

\begin{figure}
  \centering
  \subfloat[][]{\includegraphics[width=0.5\linewidth]{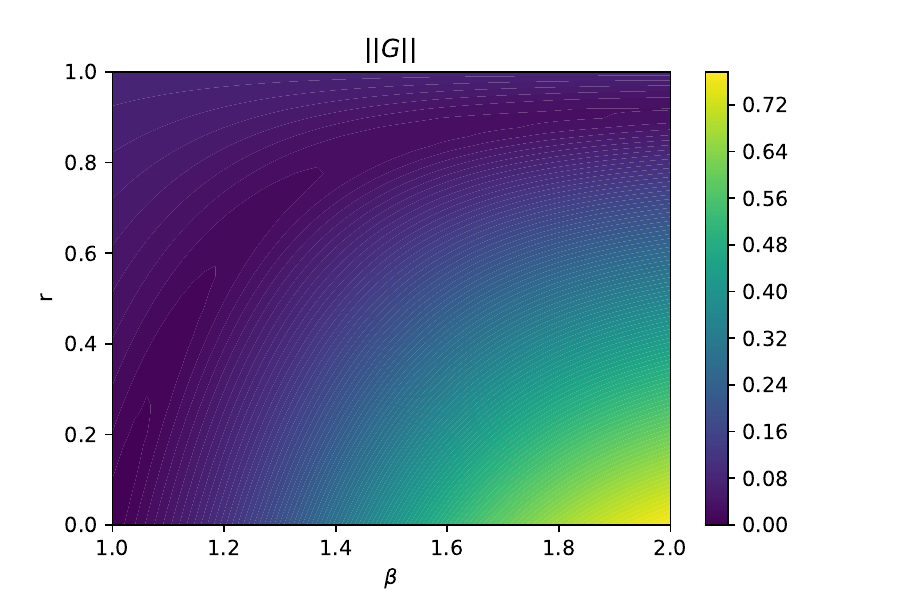}}%
  \subfloat[][]{\includegraphics[width=0.5\linewidth]{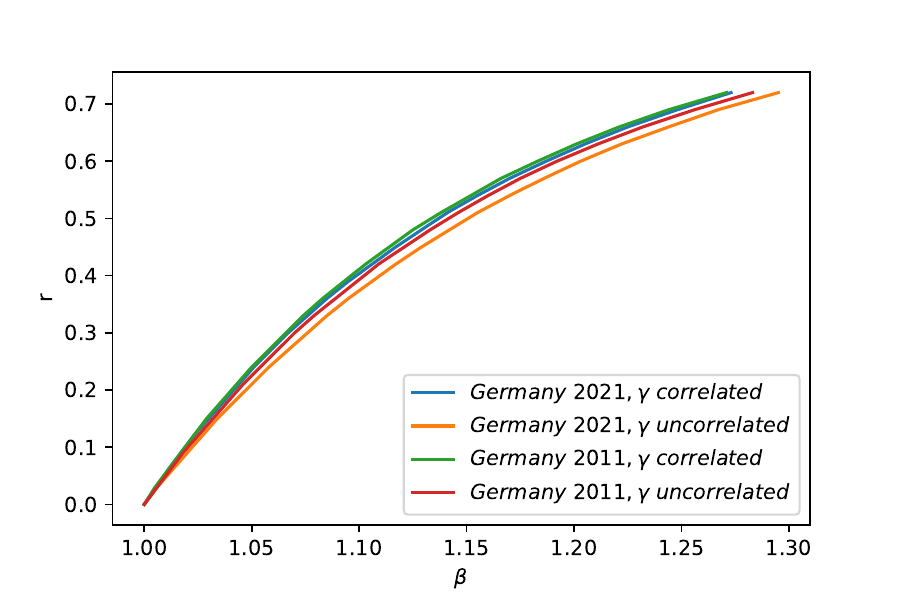}}%
  \caption{The left figure shows a contour plot of $\|G(x_{ger})\|$ with a normalized sample $x_{ger}$ from $CDF_{ger}$ \eqref{eq: CDFger} for different values of $r$ and $\beta$, where $\gamma$ and $x_{ger}$ are fully correlated. The right one shows the minimizing $r-\beta$-line for different years where $\gamma$ and $x_{ger}$ are fully correlated or uncorrelated.}
  \label{figure: r-b-line}
\end{figure}

\begin{proposition}\label{prop: rbLine}
    For any fixed $\beta\in\R, x, \gamma\in\Delta_{A-1}$, the Euclidean norm $\|G(x)\|$ is minimal when
    $$r=r^\star(\beta)\coloneqq1-\frac{\langle p(x)-\gamma, x-\gamma\rangle}{\|p(x)-\gamma\|^2}$$
    If $r^\star(\beta)<0$ (resp. $>1$), it has to be replaced by $0$ (resp. $1$). 
\end{proposition}

\begin{proof}
Define $w=p(x)-x$ and $v=\gamma-x$, such that $G(x)=(1-r)w+rv$ is a convex combination of $v$ and $w$. Then:
\begin{align*}
    \frac{d}{dr}\|G(x)\|^2&=\frac{d}{dr}\sum_{i=1}^A((1-r)w_i+rv_i)^2=2\sum_{i=1}^A((1-r)w_i+rv_i)(v_i-w_i)\\
    &=2\sum_{i=1}^A\left(r(v_i-w_i)+w_i\right)(v_i-w_i)=2r\|v-w\|^2+2\langle w, v-w\rangle\\
    &=2r\|\gamma-p(x)\|^2+2\langle p(x)-x, \gamma-p(x)\rangle
\end{align*}
Since $\|G(x)\|^2$ is a non-negative quadatic polynomial in $r$, the unique minimum is $$r=\frac{\langle p(x)-x, p(x)-\gamma\rangle}{\|p(x)-\gamma\|^2}=1-\frac{\langle x-\gamma, p(x)-\gamma\rangle}{\|p(x)-\gamma\|^2}\,.$$
\end{proof}

$r^\star(\beta)$ can be interpreted as the orthogonal projection of  $x-\gamma$ on $p(x)-\gamma$. When $x-\gamma$ and $p(x)-\gamma$ are negatively correlated, i.e. $\langle x-\gamma, p(x)-\gamma\rangle<0$, then $r=1$ is optimal. If $\|p(x)-\gamma\|<\|x-\gamma\|$ and the angle between $x-\gamma$ and $p(x)-\gamma$ is small, then $r=0$ is optimal. Moreover, if $(x_i, \gamma_i)_{i\in[A]}$ is a normalized sample of a positive random vector $(X, \Gamma)$, then
$$r^\star(\beta)\approx1-\frac{Cov\left(\frac{X}{\E X}-\frac{\Gamma}{\E \Gamma}, \frac{X^\beta}{\E X^\beta}-\frac{\Gamma}{\E\Gamma}\right)}{\E\left(\frac{X^\beta}{\E X^\beta}-\frac{\Gamma}{\E \Gamma}\right)^2}$$
for large $A$. Hence, $r^\star(\beta)$ is asymptotically (for $A\to\infty$) independent of $A$.


Figure \ref{figure: r-b-line} (a) shows a contour plot of $\|G(x_{ger})\|$ for different choices of $r, \beta$, where $x_{ger} \in\Delta_{A-1}$ is a normalized sample from the empirical wealth distribution $CDF_{ger}$ \eqref{eq: CDFger}. This indicates that the relation between the parameters is indeed one-by-one, i.e. for any given $r$ there is exactly one optimal $\beta$. Figure \ref{figure: r-b-line} (b) underlines, that the resulting r-$\beta$-line is fairly stable in time and not very sensitive on the correlation between wage $\gamma_i$ and wealth $x_i$. The lines are very similar when wage and wealth is assigned independently or when they are fully correlated. Moreover, the point $\beta=1,\, r=0$ lies on the curve of optimal points as $G$ vanishes in that case. Since we already fixed $r=0.3$ in Subsection \ref{subsec: r}, the minimum $\|G(x_{ger})\|=0.0006$ is attained for $\beta=1.068$. For the following, we consider the rounded value $\beta=1.1$ as an appropriate choice, which is also consistent with previous independent estimates of this reinforcement parameter \cite{Vallejos, forbes}.

\subsection{Simulation results}\label{subsec: sim}

Figure \ref{figure: Simulations} (a)  shows the results of simulations with parameters $A, n, r, \beta, \gamma$ as specified above and several initial configurations. For comparison, we simulate for symmetric $X(0)=(1,\ldots, 1)$, for independent exponentially distributed $X_i(0)$, for independent Pareto distributed $X_i(0)$ (exponent 1.5) and for $X(0)=A \gamma$, i.e. $X(0)$ and $\gamma$ are fully correlated. 
In each case agents start with one unit of wealth on average. The simulated wealth distribution $CDF_{sim}$ (\ref{eq: SimulationCDF}) after $n$ steps 
is both compared to the real wealth distribution $CDF_{ger}$ in Germany 2021 (black line) and to the CDF of scaled wages $228,000\gamma$ (blue line), which describes the long-time wealth distribution in a hypothetical world with constant return rates ($\beta=1$ ).

\begin{figure}
  \centering
  \subfloat[][]{\includegraphics[width=0.5\linewidth]{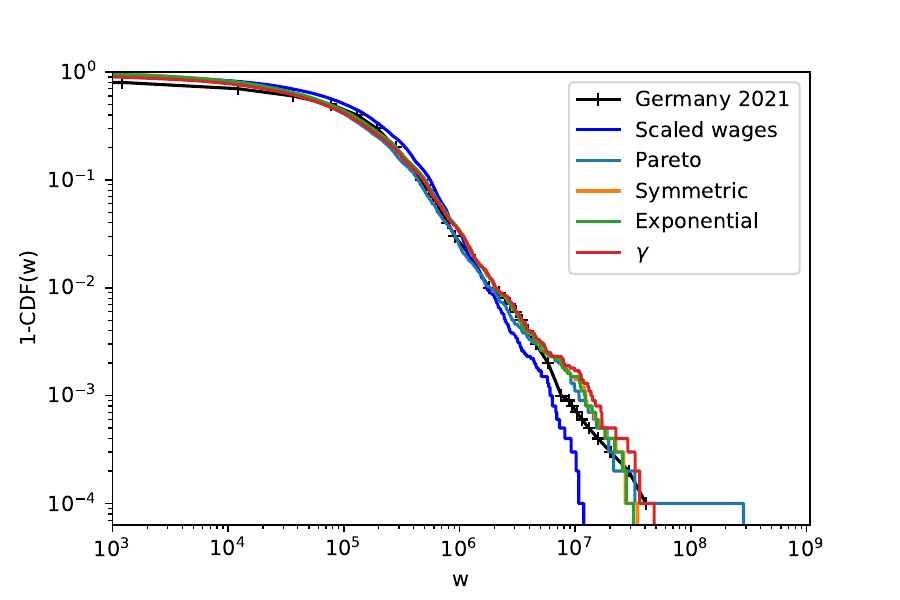}}%
  \subfloat[][]{\includegraphics[width=0.5\linewidth]{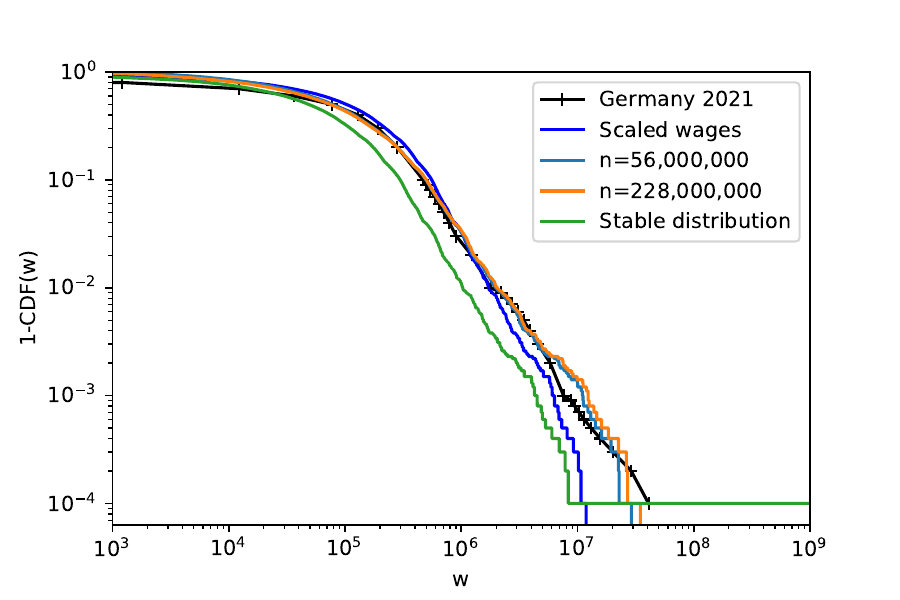}}%
  \caption{(a) shows $1-CDF_{sim}$ \eqref{eq: SimulationCDF} for different initial distributions compared to $1-CDF_{ger}$ \eqref{eq: CDFger} (black line) and $1-CDF$ of scaled wages $228,000\gamma$ (blue line) after $n=228,000,000$ steps. For symmetric initial condition $X(0)=(1,\ldots, 1)$, (b) shows additionally 1-CDF of an intermediate step, scaled to the same mean (light blue). Moreover, the green line shows a stable distribution with $\| G(x)\| =0$ obtained using Euler's method for the ODE \eqref{eq: ODE} starting in $X(n)$. In all simulations, we took $A=10,000,\, \beta=1.1,\,r=0.3$ and $\gamma$ as explained in Subsection \ref{subsec: gamma}.}
  \label{figure: Simulations}
\end{figure}

Basically, all simulations presented in Figure \ref{figure: Simulations} (a) reveal an astonishing accuracy in reproducing $CDF_{ger}$. The Gini coefficient in our simulations varys between 0.723 (exponential case) and 0.758 (Pareto case) compared the empirical value of 0.75 in Germany 2021 (see \cite{wid}). Apart from the net worth of the richest agent, all simulations yield similar wealth distributions, which is consistent with the independence of initial configuration explained in Section \ref{sec: model}. In fact, the bottom 99\% of population is already quite well described by simply re-scaling the wage distribution, since their wealth is mainly determined by savings from wages. Apparently, the impact of increasing returns is negligible for these agents. In contrast to that, the richest percentile reveals a much greater wealth inequality in reality than in the scaled wages distribution, which is well-established for empirical data \cite{CorrelationPaper, quadrini}. The proposed reinforcement mechanism of the Pólya urn model provides an accurate explanation for this significant gap.

Noticeably, the wealth of the richest agent seems to be severely overestimated in the simulation with Pareto distributed initial configuration, whereas it is underestimated in the other simulations. This observation can be underlined by considering the share of total wealth owned by certain parts of the population. The following table shows wealth shares in our simulation in comparison to German data from \cite{wid}.

\begin{center}
    \begin{tabular}{l|c c c c c}
         Share of richest & 50\% & 10\% & 1\% & 0.1\% & 0.01\%   \\ \hline
         Germany 2021 & 96.6\% & 58.9\% & 28.6\% & 14.3\% & 8.17\%\\
         Pareto & 95.3\% & 63.1\% & 34.2\% & 19.4\% & 12.4\%\\
         Exponential & 94.5\% & 57.9\% & 24.6\% & 8.5\% & 1.4\%\\
         Symmetric & 94.7\% & 58.2\% & 24.9\% & 8.7\% & 1.5\%\\
         $\gamma$ & 95.6\% & 61.2\% & 28.2\% & 10,8\% & 2.1\%\\
    \end{tabular}
\end{center}

In fact, this ostensible discrepancy can easily be explained by the enormous variance of wealth within the group of the richest 0.01\%. For instance, this group of approximately 7,000 agents starts at a net worth of 40 million Euro and contains 138 billionaires, which own up to 40 billion Euro according to the Forbes List 2021. This billionaire effect has been investigated in detail in \cite{forbes} for UK data and is already hinted at in the tail behavior of Figure \ref{figure: WID}. This extremely heterogeneous group is only represented by one agent in our simulation, such that the wealth share of all other agents is significantly impacted by the choice of the representative of the richest group. E.g. in the Pareto case, a rather rich representative of the richest group has been chosen, leading to an overestimation of the wealth share of the richest 0.1\% since this group also contains the representative of the 0.01\% and vice versa for the other simulations. Hence, we should rather consider adjusted wealth shares, where the effect of the richest agent is erased. To be more precise, if $s(\epsilon)$ is the wealth share of the richest $A\epsilon,\,\epsilon\in\{0.1, 0.01, 0.001\},$ agents, then the adjusted share is defined by 
$$s_{ad}(\epsilon)\coloneqq\frac{s(\epsilon)-s(0.0001)}{1-s(0.0001)}\,,$$
i.e. $s_{ad}$ is the wealth share, when the richest agent is removed from the system. The following table shows the adjusted wealth shares in our simulations and in reality.

\begin{center}
    \begin{tabular}{c|c c c c c}
         Adjusted share of richest & Germany & Pareto & Exponential & Symmetric & $\gamma$  \\ \hline
         1\% & 22.2\% & 24.8 \% & 23.5\% & 23.7\% & 26.6\% \\
         0.1\% & 6.6\% & 7.9\% & 7.2\% & 7.3\% & 8.8\%
    \end{tabular}
\end{center}
With these adjusted shares, we retrieve the good accuracy of all simulations, which we already observed in Figure \ref{figure: Simulations}.\\

\begin{figure}
  \centering
  \subfloat[][]{\includegraphics[width=0.5\linewidth]{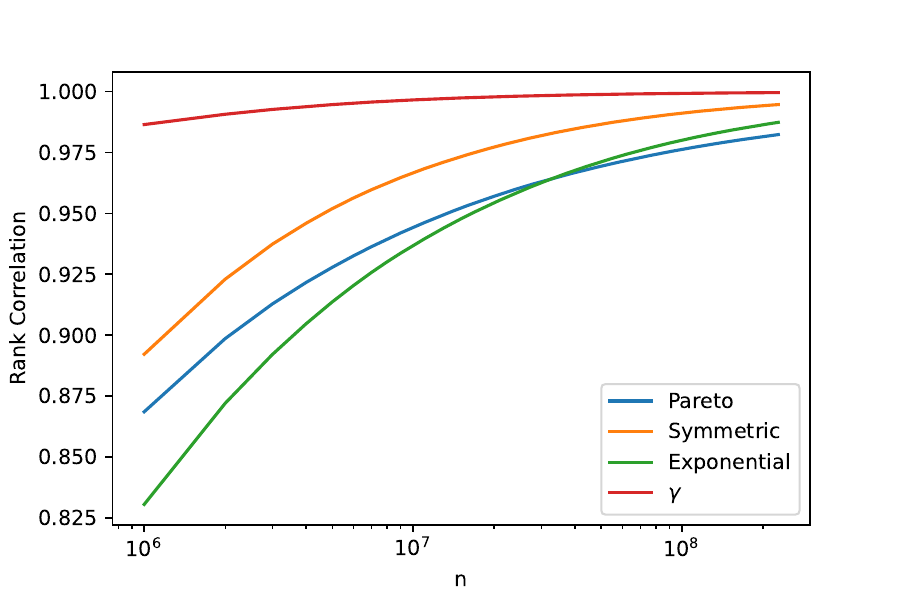}}%
  \subfloat[][]{\includegraphics[width=0.5\linewidth]{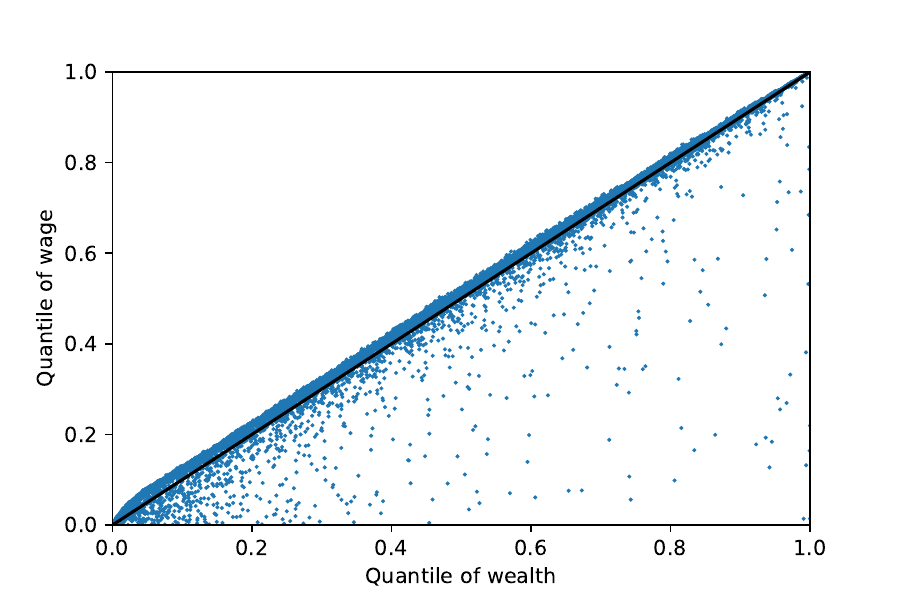}}%
  \caption{The left figure shows the evolution of rank correlation between $X(n)$ and $\gamma$ for different initial conditions. For the Pareto case, each point in the right figure represents one agent and shows its rank in $\gamma$ vs. the rank in $X(n), n=280,000,000$.
  }
  \label{figure: Rank}
\end{figure}

\textbf{Correlation between wages and wealth.} Figure \ref{figure: Rank} (a) shows the evolution of rank correlation between $X(n)$ and $\gamma$. In fact, the rank correlation approaches one in  all simulations, even when $X(0)$ is drawn independently of $\gamma$. Hence, our process is strongly ordering and the agents with high wage tend to become the richest agents. 
This feature of our model is opposed to empirical findings, e.g. \cite{CorrelationPaper} provides comprehensive statistical information about the common distribution of wealth and income in Germany and estimates a much weaker rank correlation of 0.49 between income and net worth. \cite{CorrelationPaper} mentions the importance of splitting up wealth through inheritance as a main reason for this moderate correlation, which is not captured by our model. The more-generation model presented in \cite{benhabib2} particularly focuses on the impact of inheritance. Moreover, there is an intrinsic difference between \cite{CorrelationPaper} and our correlation, because our $\gamma$ rather represents savings than income, which obviously increases the correlation with wealth. Nevertheless, agents, who start with large initial wealth, may keep their advantage in our model for quite a long time as shown in Figure \ref{figure: Rank} (b). Even after $280,000,000$ million steps, there are some agents with only little wage among the richest. This barely occurs for more equal initial configurations. On the other hand, agents cannot have high wages and only little wealth since our process is a pure growth process and wages are assumed to be deterministic.\\

\begin{figure}
  \centering
  \subfloat[][Total population]{\includegraphics[width=0.5\linewidth]{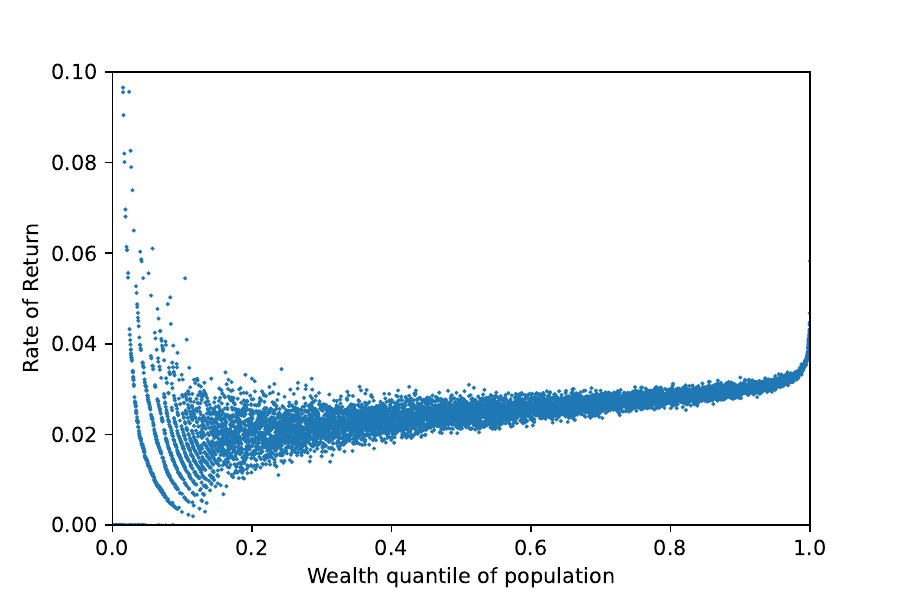}}%
  \subfloat[][Top 10\%]{\includegraphics[width=0.5\linewidth]{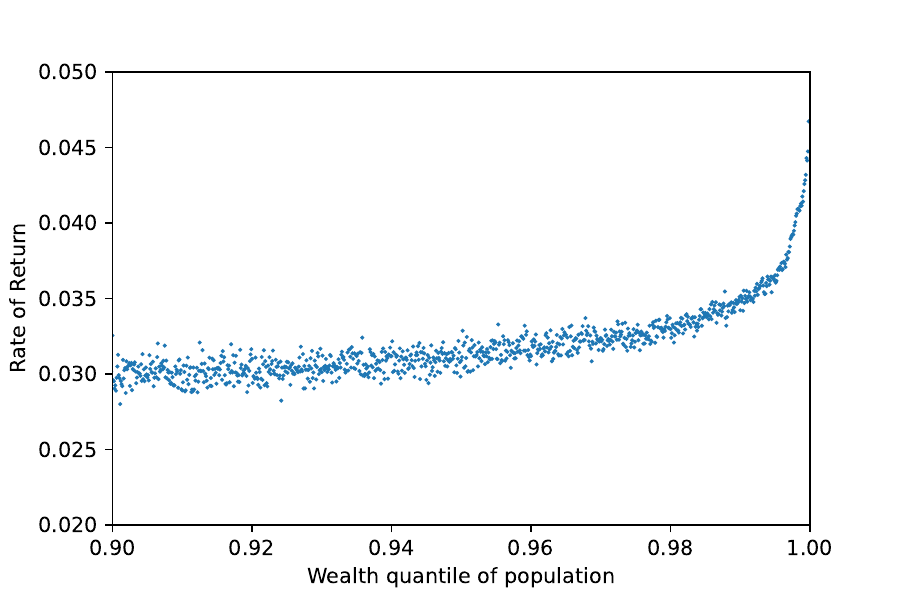}}%
  \caption{Rate of Return on capital between step $n=216,000,000$ and $n=228,000,000$ in the Pareto case plotted against wealth quantiles. This corresponds to the year 2020.}
  \label{figure: RoR}
\end{figure}

\textbf{Increasing returns.} As pointed out before, the generalized Pólya urn model with $\beta>1$ implements the idea of increasing returns, where capital return rates are higher for richer agents. But how does this dependency look like in detail? In order to compute capital return rates in our model, we identify step 216,000,000 with the year 2020, corresponding to an average wealth of 216,327 Euro in 2020. Then we define the rate of return of agent $i\in[A]$ in 2020 as the relative gain in wealth which is not due to wages received,
$$RoR_i\coloneqq\frac{X_i(228,000,000)-X_i(216,000,000)-12,000,000\cdot r\gamma_i}{X_i(216,000,000)} \ .
$$
Figure \ref{figure: RoR} (a) shows these RoR for all agents in the Pareto case, but corresponding plots for the other simulations look essentially the same. For the bottom 10\%, we oberserve an enormous variance of the RoR. Whereas many agents had no capital returns at all, some others won one or two  steps of the process by luck, leading to large RoRs due to their low level of wealth. This also explains the observable stratified shape of the wealth-RoR-plot for the bottom $20\%$. For the broad middle class in our simulation, we observe a moderate variance of RoR and only a slight dependence on wealth. Hence, increasing returns mainly concern the rich. Figure \ref{figure: RoR} (b) shows the RoR of the top 10\% in detail. A significant increase of RoR can only be detected for the top 1\% of the agents. This is consistent with the observation from Figure \ref{figure: Simulations} that the wealth distribution of the bottom 99\% is almost equivalent to the "scaled wages"-distribution. Using data from Norway, \cite{fagereng} empirically investigates the dependence of return rates and wealth, which reveals a similar shape of the wealth-return-curve. Moreover, they emphasize that this shape is persistent in time apart from extreme events like the financial crisis, where even decreasing returns could be observed. Almost constant returns for the majority of the population and strongly increasing return rates for the top induce the two-tailed wealth distribution mentioned in the introduction.\\

\subsection{Time evolution and predictions for the future}\label{sec: future}

Our model seems to be appropriate for describing the present, but does it also reproduce the empirical wealth dynamics of the last decades? And if yes, what does it predict for the future? Since all our simulations are equally valid, we focus in this subsection on the one with symmetric initial condition, which is shown in Figure \ref{figure: Simulations} (b). We observe that the wealth distribution after 1/4 of the steps is quite similar to the final distribution, so the attained wealth distribution is fairly stable at least on a moderate time horizon. To identify the steps of our process with years in reality, we consider in this chapter the time changed process 
\begin{equation}\label{eq: realtime}
    t\mapsto Z^{(N)}(t):=\chi \left(\lfloor \left((1+\mu)^t-1\right)N\rfloor\right)\,,
\end{equation}
where $t\geq 0$ is the \textbf{real time} measured in years and $\mu=\mu(t)$ is the annual growth rate of our economy. For example step $n=96,000,000$ represents the year 1995 supposing empirical growth rates from \cite{wid}.  This coincides with an average wealth per agent of approximately 96,000 Euro in 1995. Despite some historical shocks, assuming constant growth $\mu=0.03$ is a good approximation over the last 100 years (see Figure \ref{figure: wealth_timeline}), such that we will use this assumption for our future predictions.

\begin{figure}
  \centering
  \subfloat[][Share of richest 1\%]{\includegraphics[width=0.5\linewidth]{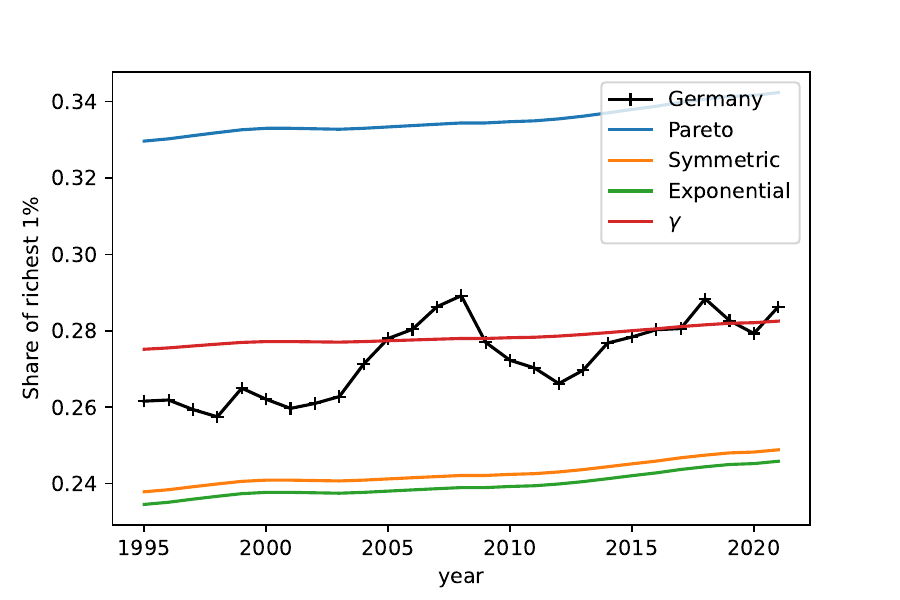}}%
  \subfloat[][Share of richest 10\% to 1\%]{\includegraphics[width=0.5\linewidth]{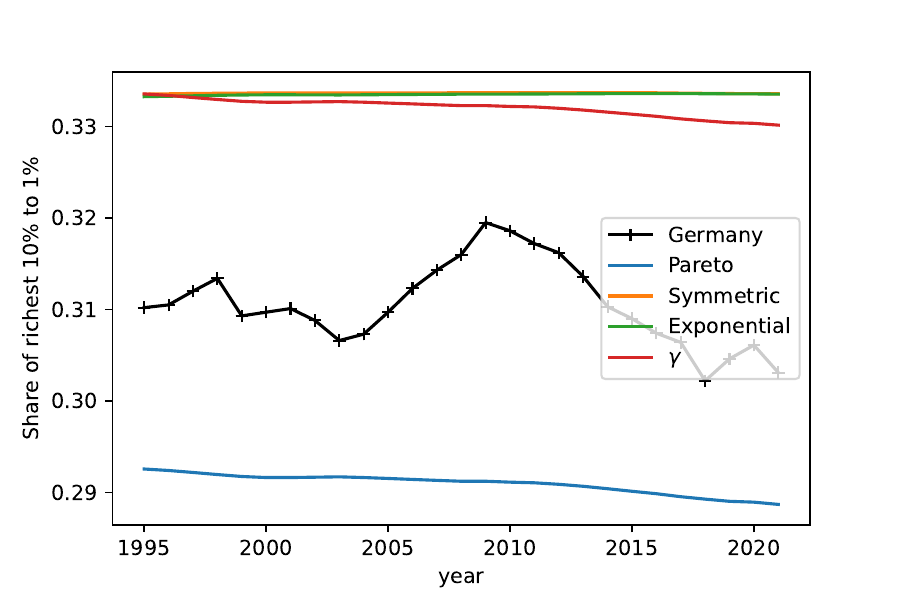}}%
  \caption{The evolution of the wealth shares of the richest 1\% and the following 9\% in our Simulations compared to German data from \cite{wid}. Years have been assigned via (\ref{eq: realtime}) using empirical $\mu$ from \cite{wid}.}
  \label{figure: Share}
\end{figure}

Figure \ref{figure: Share} (a) presents the evolution of the wealth share of the top 1\% in our simulations and in reality, all of which reveal a moderate increase. The small differences in the total level of that share have already been discussed above. It should also be noted that the development is much smoother in our simulations than in reality. This is again due to the fact that our model does not encompass the impact of economic shocks. The financial crisis in 2008, for example, led temporarily to a decreasing share of the richest due to falling stock and real estate markets. Figure \ref{figure: Share} (b) shows the wealth share of the 10-1\% quantile, which is almost stagnant with a slightly decreasing trend. Hence, only the richest managed to slowly improve their position over the last decades at the expense of the middle class and even the "moderate" upper class, in reality as well as in our simulations. Consequently, our model did also accurately reproduce the past dynamics of wealth, which justifies to use our model for future predictions. \\

 In Appendix \ref{sec: llnTax}, we present a functional law of large numbers for our process, stating that the process is  asymptotically deterministic for large initial values and the dynamics are driven by the field $G$ \eqref{eq: G}, representing the expected increments of wealth shares up to a time scaling \eqref{eq: G2}. To be more precise, for large enough $N$ the process $Z^{(N)}$ is well approximated by the solution of the ODE
\begin{equation}\label{eq: ODE2}
        \frac{d}{dt}Z(t)=G(Z(t))\ln (1+\mu)\approx\mu G(Z(t))\quad\text{with } Z(0)=Z^{(N)}(0)\,.
\end{equation}
This is an efficient tool to make predictions for the future by solving (\ref{eq: ODE2}) with our simulation's result as initial condition, using e.g. Euler's method. Since the number of steps per year in the urn model increases exponentially with \eqref{eq: realtime} simulating the model is computationally much more demanding than simply solving (\ref{eq: ODE2}) numerically. 
Figure \ref{figure: G_Evolution} (a) compares these predictions to our simulations with very good agreement for large initial data. For small or moderate initial data fluctuations play a significant role in the stochastic evolution of the urn model. 
Of course, these predictions suppose that the dynamics of wealth remain unchanged in the future, which might not be the case due to the recent increase of interest rates. We will return to this issue in Section \ref{sec: discussion}.

\begin{figure}
  \centering
  \subfloat[][Simulation]{\includegraphics[width=0.5\linewidth]{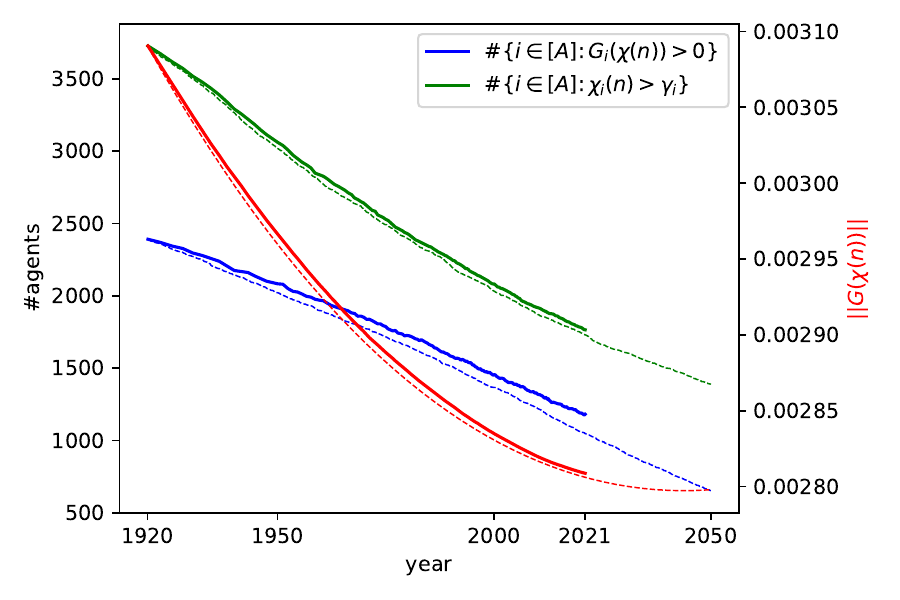}}%
  \subfloat[][Prediction]{\includegraphics[width=0.5\linewidth]{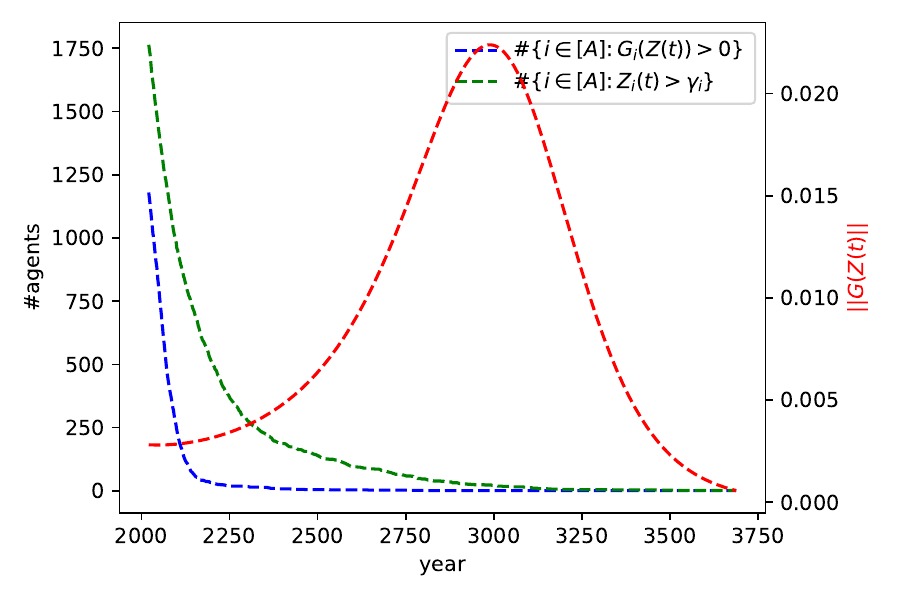}}%
  \caption{(a) shows the evolution of number of winners and agents with positive $G_i(\chi(n))$ as well as $\|G(\chi(n))\|$ in the simulation with $X(0)=(1,\ldots, 1)$. Years are assigned via (\ref{eq: realtime}) assuming a constant growth rate of $\mu=0.03$ per year. The dashed lines in (a) show numerical solutions of (\ref{eq: ODE2}) started from simulation data in 1920, which agree well with the time evolution of simulation data. In (b) solutions of  (\ref{eq: ODE2}) are shown starting from simulation data in 2021.}
  \label{figure: G_Evolution}
\end{figure}

\noindent Now, we consider the past and future time evolution of three indicators, presented in Figure \ref{figure: G_Evolution}:\\

The \textbf{number of agents with positive field $G_i$.} 
According to \eqref{eq: ODE2}, $G_i(\chi(n))$ is a good indicator for the short-term development of the share of agent $i\in[A]$, where positive (negative) values indicate an increasing (decreasing) share. $G_i(\chi(n))>0$ can occur in two different ways: either by large expected capital returns due to large $\chi_i(n)$ or by large wages compared to $\chi_i(n)$. After $n=280$ million steps (corresponding to the year $t=2021$), only  11.8\% of all agents have positive  $G_i(\chi (n))>0$. 48\% of the richest decentile  and even 97\% of the richest percentile belong to this group, whereas only 4\% of the poorer half of the population do so. Figure \ref{figure: XGplot} presents a detailed scatter plot of $G_i(\chi(n))$ and $\chi_i(n)$ for two different $n$.  As a consequence, most rich agents will increase their wealth share further, whereas the majority of the population loses. As visible in Figure \ref{figure: G_Evolution} (a), the number of agents with positive short-term trend has been decreasing in time in our simulation. Figure \ref{figure: G_Evolution} (b) presents a long time prediction for this indicator. The number of agents with positive trend will further decrease until only one agent is left, but this would take another 785 years.

The \textbf{number of winners}. 
In Section \ref{sec: model}, we referred to an agent $i\in[A]$ as winner if their wealth share exceeds the wage share, i.e. $\chi_i(n)>\gamma_i(n)$. In that sense, we can identify 18,87\% of agents as winners in our simulation in 2021. This group consists exclusively of agents belonging to the bottom 15\% or top 5\%. The high amount of poor agents in this group is due to the symmetric initial condition, where agents with low wage start with relatively large wealth share. As visible in Figure \ref{figure: G_Evolution}, the number of winners has been strongly decreasing in time and will further decrease until only one winner is left, but on an even longer time horizon than the previous indicator.

The \textbf{norm $\|G(\chi(n))\|$.} 
According to (\ref{eq: ODE2}), $\|G(\chi(n))\|$ can be considered as a measure for the local pace of expected change.  In Figure \ref{figure: G_Evolution}, we observe decreasing $\|G(\chi(n))\|$ in our simulation and we know from Theorem \ref{thm: intax} that the norm vanishes asymptotically as we approach a fixed point of $G$. Following our prediction, it will reach a local minimum in 20 years, followed by a strong increase, which will last over 1,000 years in theory. Finally, it converges exponentially towards zero. Recall that our reinforcement parameter $\beta=1.1$ was chosen such that $\|G(x_{ger})\|$ is small (c.f. Subsection \ref{subsec: beta}). In order to gain an intuition for this behaviour, we refer the reader back to the 3-agents case discussed in Figure \ref{figure: field} (a). Following a typical trajectory starting near the center of the simplex, it will first approach one of the unstable fixed points, before it finally turns towards a stable fixed point. On this trajectory, the number of winners and agents with $G_i>0$ is decreasing. Consequently, the observed current local minimum of $\|G(\chi(n))\|$ indicates that our simulated economy is currently close to an unstable fixed point. It may remain near this fixed point for some time, but the dynamics of our model will eventually accelerate and lead the economy into a monopoly-like state.

In this final point, the richest agent dominates the market with a share of 45.8\%. Figure \ref{figure: Simulations} (b) shows the corresponding \textbf{stable wealth distribution}, which is defined  analogously to (\ref{eq: SimulationCDF}) for any stable fixed point $x\in\Delta_{A-1}$ (with $228,000x_i$ in the place of $10 X_i(n))$. So even within our model, which does not take into account any future changes of parameters or the fundamental mechanism of the dynamics, it would take many centuries until such a monopoly-like state is attained.


\section{Unequal investment skills as an alternative explanation}\label{sec: alpha}

Throughout our considerations in Section \ref{sec: simulation}, we assumed that return rates on capital do only depend on wealth, but not on individual skills. Nevertheless, it is conceivable that unequal investment skills pose an alternative (or additional) explanation for the gap between wage and wealth distribution (empirically found in e.g. \cite{quadrini, CorrelationPaper}). This question can also be discussed within our extended Pólya urn model. For that, we set $F_i(k)=\alpha_i k^\beta$, where $\alpha_i>0$ regulates the investment skills of agent $i$. We keep the parameters $A=10,000,\, r=0.3$ and $\gamma$ as in Section \ref{sec: simulation} since they were derived without using the assumption of equal $\alpha_i$.

For specifying $\alpha_i$, we suppose that there is a positive correlation with wages. For simplicity we use the ansatz
$$\alpha_i=\gamma_i^c\quad\text{for some }c\ge0\,,$$
where $c$ regulates the intensity of correlation between wages and investment skills. In particular, $c=0$ corresponds to the the equal skill case from Section \ref{sec: simulation}, whereas large $c$ implies huge differences in investment skills. Note that the vector $(\alpha_1,\ldots, \alpha_A)$ does not need to be normalized since only ratios of $\alpha_i$ enter the dynamics \eqref{eq: transpr}.

In order to find an appropriate $\beta$ for this situation, we have another look at the $r-\beta$-line derived in Proposition \ref{prop: rbLine}, which gives the pairs $r, \beta$ minimizing $\|G(x_{ger})\|$ with a normalized sample $x_{ger}$ from $CDF_{ger}$ (\ref{eq: CDFger}). Hence, we choose again our parameters such that the empirical wealth distribution $CDF_{ger}$ in Germany for 2021 is as close to a stable distribution in our model as possible. Figure \ref{figure: Alpha} (a) shows this $r-\beta$-line for several choices of $c$. First, it is immediately noticeable that positive $r>0$ is optimal for $\beta=1$ and any $c>0$. To get an intuition on this, recall that for $\beta=1$, $r=0$ and $c>0$ our process reveals weak monopoly (see Appendix \ref{sec: r=0}). Since the real wealth distribution is of course more equal than weak monopoly, we need to choose a positive labor share. Second, the optimal labor share is increasing in $c$. This is due to the fact that larger $c$ increases inequality in our model, which can be compensated by larger $r$. Third, for fixed $r$ there is an inverse relation between $c$ and $\beta$, because they both increase inequality. Note that for large $c$ even $\beta<1$ is optimal, which corresponds to decreasing return rates. We first focus on a model with $c>0$ and $\beta =1$ considering only investment skills and no reinforcement, in order to highlight the conceptual differences to the situation examined in Section \ref{sec: simulation}. 
This case has been rigorously treated in Proposition \ref{prop: alpha}, where we proved that the shares $\chi(n)$ converge to a deterministic point for $\beta=1$ and $c>0$.

\begin{figure}
  \centering
  \subfloat[][]{\includegraphics[width=0.5\linewidth]{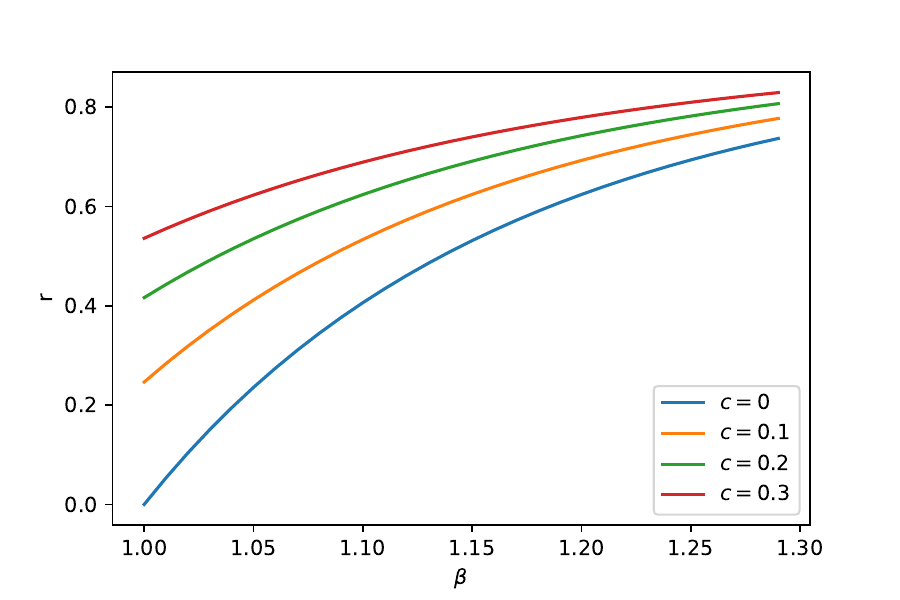}}%
  \subfloat[][]{\includegraphics[width=0.5\linewidth]{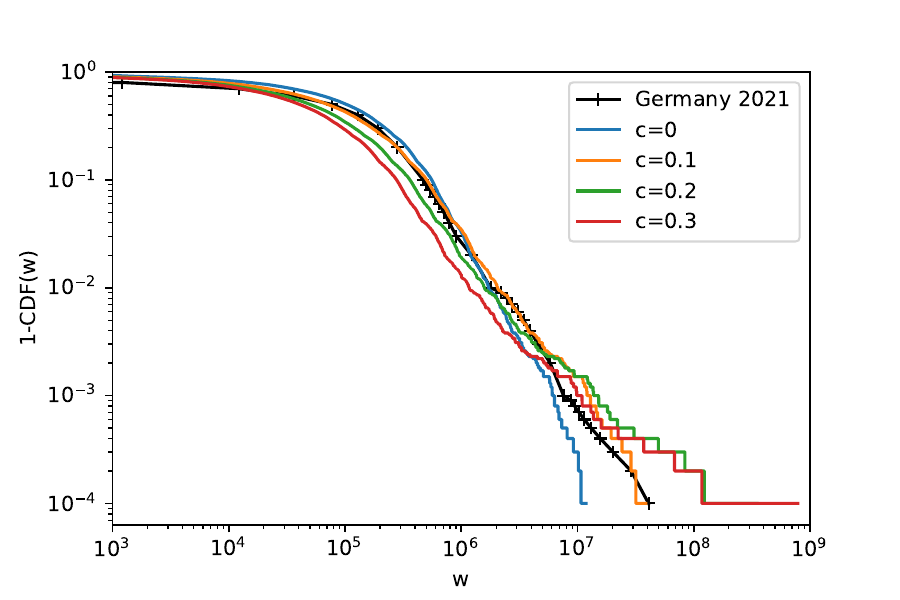}}%
  \caption{On the left, we see the minimizing $r-\beta$-line as derived in Proposition \ref{prop: rbLine} using data from Germany 2021 and sorted $\gamma$, but varying $c$. On the right, we see the stable wealth distributions obtained by Euler's method for \eqref{eq: ODE} with $\beta=1,\,r=0.3$ and different $c$. Note that the blue line coincides with the scaled wage distribution.}
  \label{figure: Alpha}
\end{figure}

 Figure \ref{figure: Alpha} (b) shows the unique stable distribution for different $c$. The model provides a quite good approximation of $CDF_{ger}$ for $c=0.1$, too. Larger $c$ implies an overestimation of the tail weight, which is consistent with the $r-\beta$-line. Nevertheless, there is a major difference compared to the results presented in Subsection \ref{subsec: sim}, where the wealth of the richest agent is much larger and more realistic. As visible in Figure \ref{figure: Simulations} (right), the model even predicts that the wealth of the richest will further increase in the future, which is not the case for $\beta =1$, since the distribution in Figure \ref{figure: Alpha} (b) is already stable. This can be underlined by a quick look at the wealth shares, as shown in the following table.

 \begin{center}
    \begin{tabular}{l|c c c c c}
         Share of richest & 50\% & 10\% & 1\% & 0.1\% & 0.01\%   \\ \hline
         Germany 2021 & 96.6\% & 58.9\% & 28.6\% & 14.3\% & 8.17\%\\
        $c=0.1$ & 95.4\% & 59.3\% & 25.7\% & 9.4\% & 1.7\%\\
        $c=0.17$ & 96.4\% & 68.0\% & 40.5\% & 24.3\% & 9.0\% \\
        $c=0.2$ & 96.7\% & 71.1\% & 46.5\% & 31.9\% & 15.3\%\\
        $c=0.3$ & 97.3\% & 77.2\% & 58.6\% & 48.0\% & 22.4\%
    \end{tabular}
\end{center}

Indeed, $c=0.1$ significantly underestimates the wealth of the richest. If we slightly increase $c$, such that the share of the richest $0.01\%$ coincides with our data, then our model significantly overestimates the $1\%$ and $10\%$ share. Thus, the model with $\beta =1$ cannot properly reproduce the empirical wealth distribution.

This observation is linked to a conceptual difference between the two models. Whereas the process reveals a random limit in the situation of Subsection \ref{subsec: sim} with $\beta >1$, there is a deterministic limit point for $\beta =1$, i.e. the long-time limit is fully determined by skills. Hence, under increasing returns with $\beta >1$ the long time limit is affected by the initial wealth of agents, whereas it is not for constant returns with different investment skills. Moreover, the rank correlation of wage and wealth will reach one after finitely many steps with probability one for $\beta=1, \,r>0$, but not necessarily for $\beta>1$.

\begin{figure}
  \centering
  \subfloat[][$r=0.3$, $\beta^\star=1.068$, $c^\star=0$,\\minimal value: $\|G(x_{ger})\|=0.0006$]{\includegraphics[width=0.5\linewidth]{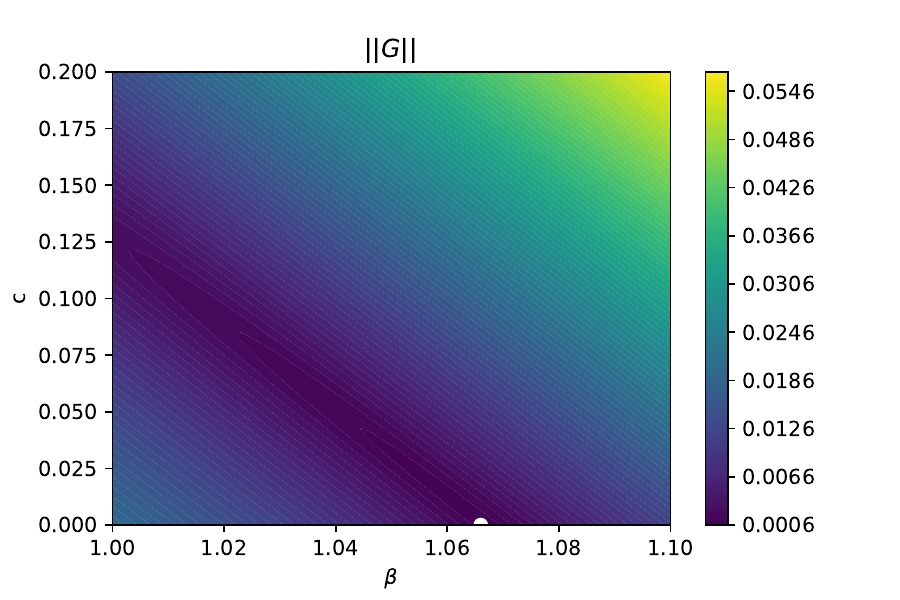}}%
  \subfloat[][$r=0.7$, $\beta^\star=1.147$, $c^\star=0.216$,\\ minimal value: $\|G(x_{ger})\|=0.002$]{\includegraphics[width=0.5\linewidth]{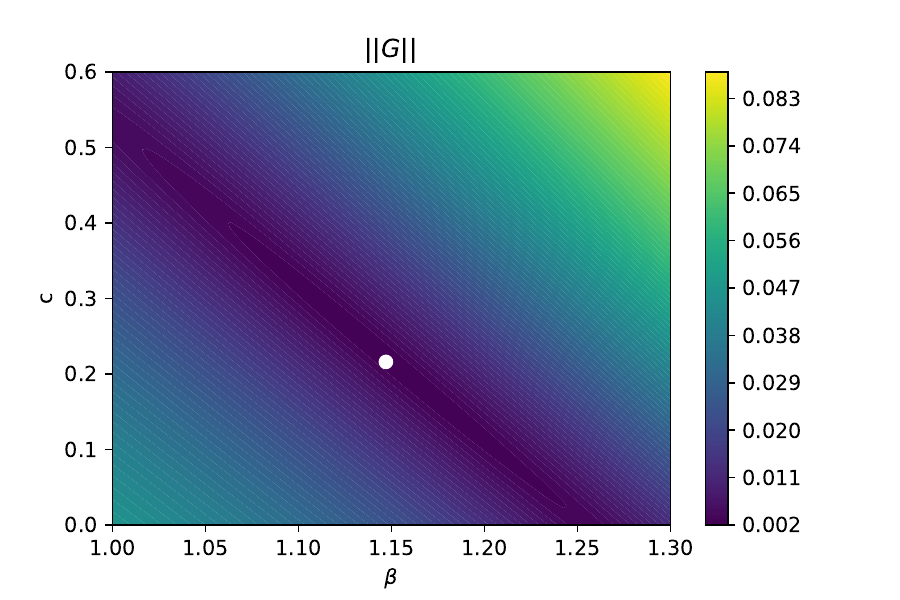}}%
  \caption{Contour-plot of $\|G(x_{ger})\|$ for varying values of $\beta$ and $c$, where $r=0.3$ resp. $r=0.5$ and $x_{ger}$ are fixed. The white bullet marks the global minimum $(\beta^\star,\,c^\star)$. Note the different scales in both plots.}
  \label{figure: Contour-r-c}
\end{figure}

It stands to reason that reality is a mixture of both, increasing returns and unequal skills. Figure \ref{figure: Contour-r-c} illustrates the goodness of fit for several choices of $c$ and $\beta$, again measured by $\|G(x_{ger})\|$ like in Subsection \ref{subsec: beta}. It underlines the reverse relation between $c$ and $\beta$, which means that in the optimum larger $\beta$ corresponds to smaller $c$ and vice versa. It turns out, that any positive $c$ provides no improvement with respect to this criterion for $r=0.3$. The reverse relation does also hold for larger $r$ (see  Figure \ref{figure: Contour-r-c} right) since the optimal $\beta$-$c$-line just shifts away from the origin. This is due to the fact that larger $c$ and $\beta$ increase inequality, but larger $r$ decreases inequality. Nevertheless, for larger $r$ one can achieve slight improvements by taking positive $c$, since too large $\beta$ leads to a overestimation of the wealth of the richest agent, in the sense that the losers basically only get their wage. But the total goodness is clearly worse for large $r$ than for our choice $r=0.3$, supporting that this is a reasonable value.

To summarize the findings in this section, we consider the parametrization of Subsection \ref{subsec: sim} as appropriate to model the evolution of the empirical wealth distribution. In particular, the assumption $c=0$ is well justified and we can ignore differences of investment skills within our model.

\section{Summary and Discussion}\label{sec: discussion}

\begin{figure}
  \centering
  \subfloat[][$r=0.3$]{\includegraphics[width=0.5\linewidth]{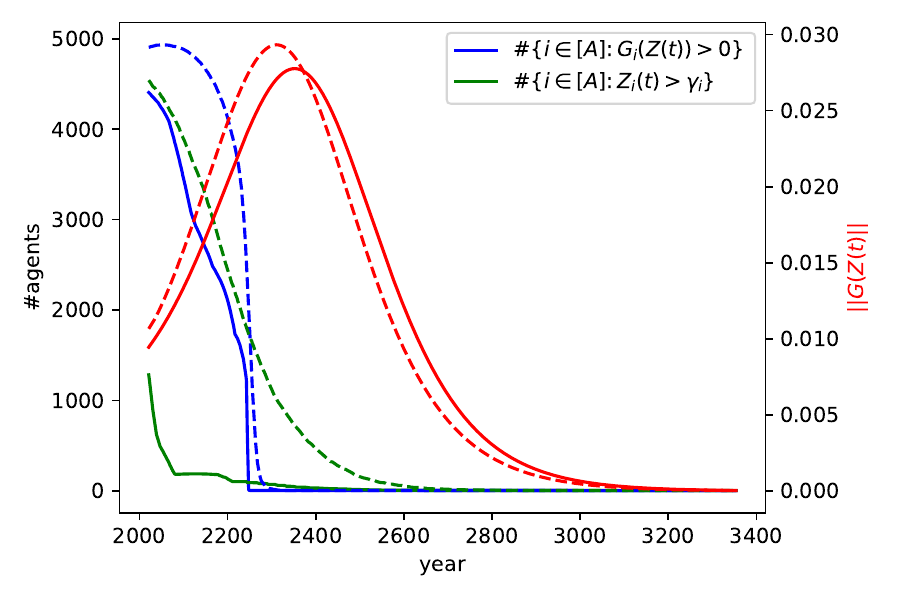}}%
  \subfloat[][$r=0.4$]{\includegraphics[width=0.5\linewidth]{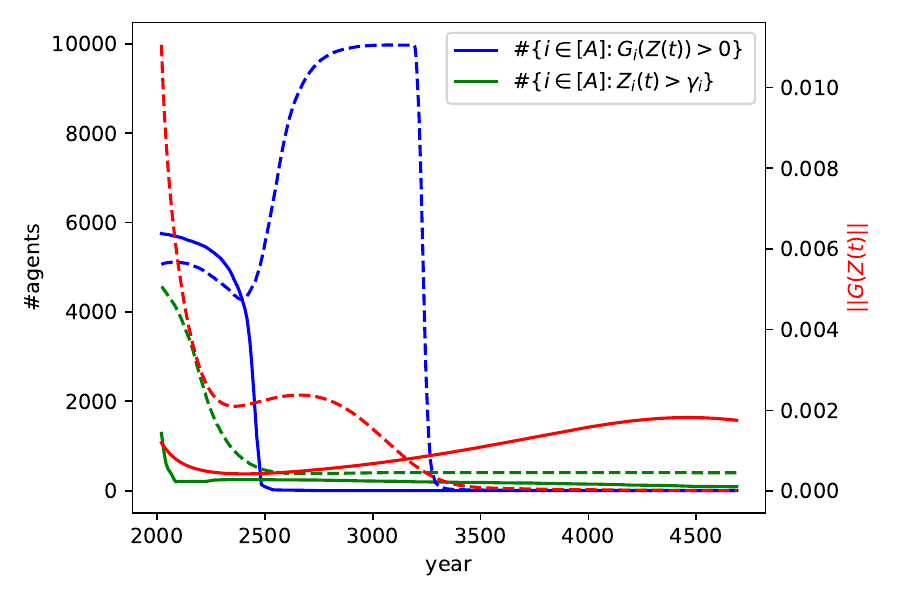}}\\
  \subfloat[][$r=0.5$]{\includegraphics[width=0.5\linewidth]{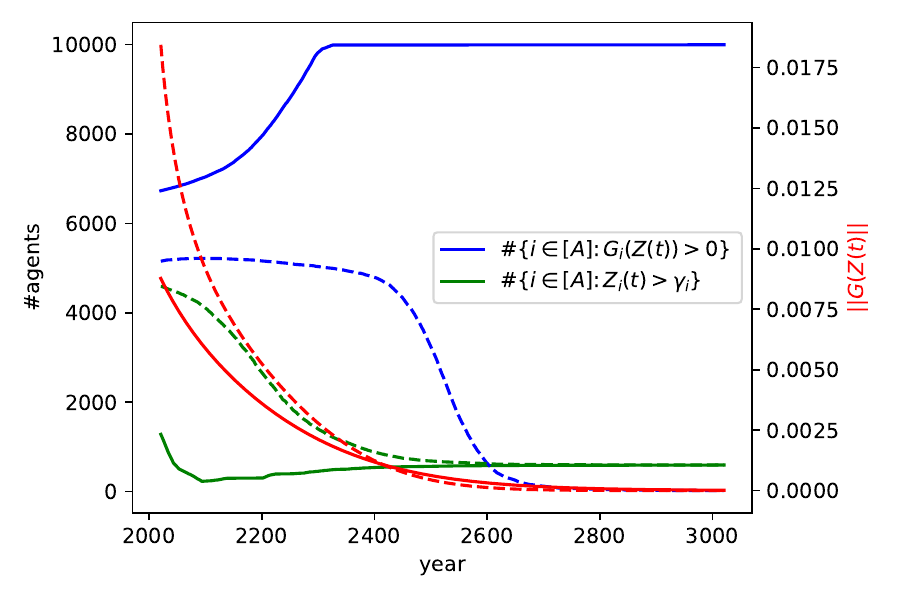}}%
  \caption{Number of winners (green), agents with positive $G_i(Z(t))$ (blue) and $\|G(Z(t))\|$ (red) in the solution of (\ref{eq: ODE2}) with initial condition $x_{ger}$ (after normalization) for different labour shares $r$. For the full (dashed) lines, wage and wealth was assigned fully correlated (uncorrelated). Note the different scales on the time axes.}
  \label{figure: EulerPredictions}
\end{figure}

In Section \ref{sec: model}, we discussed that our model basically exhibits three different regimes. First, for small $r<r_c'$ there is one random winner, who dominates the population on the long run. All agents have a positive probability of being that winner, which depends on the initial configuration and the wage distribution. Second, for large $r>r_c$, the process converges to a deterministic stable distribution, which is basically a distortion of the wage distribution towards more inequality (for $\beta>1$). Third, for moderate $r\in(r_c', r_c)$ there is still a random leading agent, but not all agents can be the leader depending on their wage. To check which regime holds for our choice of parameters we can again compute stable fixed points of the field $G$ by numerically solving \eqref{eq: ODE2} with different initial conditions. Stability of the generated fixed points was checked with the heuristics from Appendix \ref{sec: heuristics}. Assume w.l.o.g. $\gamma_1\le\gamma_2\le\ldots$, where $\gamma_1=0$ and $\gamma_A=0.0052$ in our case. Then the solution of \eqref{eq: ODE2} with initial condition $e^{(1)}$ converges towards a fixed point $(x_1,\ldots, x_A)\in\Delta_{A-1}$ with $x_1=0.451$ and $x_A=0.0042$. Hence, our process with $r=0.3$ seems to be in the first regime, where even agents with low wage can win the process when they start from a high wealth share. 
For $r=0.4$ the numerical solution finds monopoly fixed points for the richest agents but not for the poor, i.e. the it converges to a point $x\in\Delta_{A-1}$ with $x_1=0$ and $x_A=0.016$ when it starts in $e^{(1)}$. But with starting point $e^{(A-1)}$, it converges to a point with $x_{A-1}=0.183$ and $x_A=0.0074$, corresponding to the monopoly fixed point of agent $A-1$. Hence, the middle regime applies for $r=0.4$. Finally, for $r=0.5$, solutions of \eqref{eq: ODE2} converge to the same fixed point when starting in $e^{(1)}$ and $e^{(A)}$, such that the process is in the deterministic regime. In this unique stable fixed point $x\in\Delta_{A-1}$ we have $x_1=0$ and $x_A=0.0086$, which is consistent with Proposition \ref{prop:det}. In summary, we get the rough estimate
$$0.3<r_c'<0.4<r_c<0.5$$
for our situation. As a consequence, even moderately larger $r>0.3$ can lead to a different regime and long-time behaviour of our process. Recall that our choice $r=0.3$ is closely linked to a zero-interest economy (see Subsection \ref{subsec: r}). Higher interest rates might lead to a larger labor share $r$ and can therefore significantly change our predictions for the future evolution of wealth distribution. 

Analogously to Figure \ref{figure: G_Evolution} (b), Figure \ref{figure: EulerPredictions} presents the predictions of our model for different labor share $r$, again by solving (\ref{eq: ODE2}). The initial condition is the empirical wealth distribution $x_{ger}$ of 2021, assigned fully correlated with wages or uncorrelated. For the case $r=0.3<r_c'$ (Figure \ref{figure: EulerPredictions} (a)), we observe basically the same as in Figure \ref{figure: G_Evolution} (b), where we took the final state of our simulation as starting point. The time when the monopoly-like state is attained strongly depends on the wealth of the richest agent, which was underestimated in our simulation (c.f. Subsection \ref{subsec: sim}). The assignment of wealth and wage is not decisive for the future development as labor plays a minor role. As opposed to that, the predictions do significantly depend on the assignment of wealth and wage for $r=0.4\in(r_c', r_c)$ (Figure \ref{figure: EulerPredictions} (b)). For fully correlated assignment, we still observe the monopoly-like behavior with only one winning agent on the long run, but the dynamics are even slower than for $r=0.3$. But for initially uncorrelated wealth and wage, we still have several winners with $\chi_i(n)>\gamma_i$ on the long run.  Moreover, the number of winners and agents with positive $G_i$ is not monotone in time. This phenomenon does also occur in the $A=3$ case (see Figure \ref{figure: field} (c)), when we follow a trajectory starting near the corner of an agent with low wage. The huge number of agents with positive $G_i$ in the uncorrelated case is due to the effect, that the poorest agent starts with a positive share, but converges to zero (Proposition \ref{prop:det}). Consistent with our theoretical considerations, the limit point is independent of the assignment of wealth and wage for $r=0.5>r_c$, but the way towards this point varies. In the correlated case, the wealth of the richest agent is distributed among all others, which explains the large number of agents with positive $G_i$. For uncorrelated assignment of wage and wealth, the redistribution is more complex with decreasing number of short-term profiteers. As expected, for $r=0.4$ and $0.5$ the system converges to a state with more than one winner (green lines). Convergence is slowest in the intermediate regime where the structure of stable and unstable fixed points is most complex. In all three regimes, the final part of the dynamics is dominated by the fraction of the richest agent, which is the slowest variable in the system.


However, as discussed before, the actual stable points of the dynamics may not be reached in reality for various reasons. While the structure of stable limit distributions changes over time due to external influences that are not included in our model, the system evolves only slowly in a complex landscape with many fixed points with unstable directions. We have seen in Figures \ref{figure: G_Evolution} and \ref{figure: EulerPredictions} that it would take hundreds of years to reach the stable distribution with model parameters fitted from today's data.\\

In summary, the proposed model provides an accurate replication of the observed wealth distribution in Germany given the distribution of wages, widely independently of the presumed initial wealth distribution. In particular, the two tail structure from Figure \ref{figure: WID} is well reproduced. There is only some discrepancy concerning the wealth of the richest agent in our model, who represents the richest 0.01\% of real population. Since there is a huge variance within the wealth of this group we would have to simulate single households to properly represent this group. This would be computationally much more costly for a rather limited gain, and  we consider this only a minor disadvantage. Moreover, the observed wealth dynamics of recent decades is reflected in our simulation, where the wealth share of the richest percent of population grows slightly at the expense of the rest, even the "moderate" upper class. According to our model, this trend will continue in the future and less and less people will profit from increasing returns. The return rates on capital are also accurately modeled by the Pólya urn mechanism, implying that increasing returns do basically only affect the richest percentile (which eventually leads to the two-tailed structure of wealth distribution).

In order to understand the generic nature of wealth dynamics and avoid overfitting the model is kept intentionally very simple, which naturally leaves space for further refinements and research questions: 
A major intrinsic disadvantage of the proposed model is that it is strongly ordering, i.e. the rank correlation between wage and capital is close to one on the long run. This does not comply with empirical data. Including inhomogeneous investment skills (represented by the parameter $\alpha$) might pose a solution to this problem, if investment skills and wages are not chosen fully correlated. Due to a lack of useful data on the correlation structure of wealth and wage, we leave this issue open for future research. 
We concluded that unequal investment skills with constant return rates provides a less accurate explanation of empirical observations, but more refined research could be done here, including a more sophisticated model for the fitness of agents. 
Finally, our heuristics on the structure and number of stable and unstable fixed points of the driving field $G$ (Section \ref{sec: A=2}) for moderate $r$ could be completed by a rigorous treatment in the future.




\appendix

\section*{Appendix}

\section{The case $r=0$}\label{sec: r=0}

\cite{wir} and the references therein provide a comprehensive study of the wage-free model with $r=0$. The following Theorem sums up the main properties.

\begin{theorem}
    Let $r=0$ and $F_i(k)=\alpha_i k^\beta$ with $\alpha_i>0$ and $\beta\in\R$.
    \begin{enumerate}
        \item Let $\beta>1$. The process reveals strong monopoly in the sense that there is only one agent winning infinitely many steps, i.e.:
        $$\P\left(\exists i\in[A]\,\exists n_0\in\N\,\forall n\ge n_0\,\forall j\ne i\colon X_j(n)=X_j(n_0)\right)$$
        The probability of being the monopolist is positive for all agents and depends on $X(0)$ and $\alpha_i$.
        \item Let $\beta=1$ and $\alpha_1=\ldots=\alpha_A$. Then $\lim_{n\to\infty}\chi(n)$ exists almost surely and has a Dirichlet distribution with parameter $X(0)$.
        \item Let $\beta=1$ and $\alpha_i>\alpha_j$ for some $i\in[A]$ and all $j\ne i$. Then $\lim_{n\to\infty}\chi(n)=e^{(i)}$ almost surely.
        \item Let $\beta<1$. Then 
        $$\lim_{n\to\infty}\chi(n)=\left(\frac{\alpha_i^{\frac{1}{1-\beta}}}{\alpha_1^{\frac{1}{1-\beta}}+\alpha_A^{\frac{1}{1-\beta}}}\right)_{i\in[A]}\quad\text{almost surely.}$$
    \end{enumerate}
\end{theorem}

\cite{wir} also studies inhomogeneous feedback functions, revealing some additional features like non-convergence or random weak monopoly. Moreover in case 1., it is possible to predict the monopolist with high probability when the initial values are large, which forms the basis of our arguments using a numerical solution for the long-time evolution of \eqref{eq: ODE2}. Moreover, an important result from \cite{Oliveira, Zhu} is the occurrence of power-law distributions.

\begin{theorem}
    Let $A=2$, $r=0$ and $F_i(k)=k^\beta$ with $\beta>1$. Then the number of steps won by the loser has a power-law distribution, i.e.:
    $$\P\left(\min\{X_1(\infty), X_2(\infty)\}>n\right)=\Theta(n^{1-\beta})$$
\end{theorem}

The result can be extended to more general feedback functions, but an extension to larger systems $A>2$ is still left to show to our knowledge.

\section{Heuristics on the stability of fixed points}\label{sec: heuristics}

In Subsection \ref{subsec: sim}, we faced the challenge to decide whether a given fixed point of the field $G$ is stable. Formally, we would have to check negative definiteness of the Hessian of the Lyapunov function (\ref{eq: lyapunov}) using e.g. Lanczos' algorithm. Since this is numerically difficult for large $A$, we are content with some heuristics. These arguments will also grant some further insight into the possible positions of stable fixed points. We recall the definition of the field $G$ in \eqref{eq: G}, \eqref{eq: gg} with feedback $F_i(k)=k^\beta$.

A fixed point $x\in\Delta_{A-1}$ of $G$ is stable if any infinitesimal exchange of mass between two agents has an inverse effect on $G$, i.e. 
\begin{equation}\label{eq: stability}
    \frac{\partial G_i(x)}{\partial x_i}- \frac{\partial G_i(x)}{\partial x_j}<0\quad\text{for all }j\ne i\,.
\end{equation}
If $\frac{\partial G_i(x)}{\partial x_i}- \frac{\partial G_i(x)}{\partial x_j}>0$ holds for one pair $i\ne j$, then any increase of $x_i$ at the expense of agent $j$ would even be reinforced by the field $G$, such that $x$ is unstable. The partial derivatives can easily be computed:

\begin{align*}
    \frac{\partial G_i(x)}{\partial x_i}&=(1-r)\frac{\beta x_i^{\beta-1}\sum_k x_k^\beta-\beta x_i^\beta x_i^{\beta-1}}{\left(\sum_k x_k^\beta\right)^2}-1=(1-r)\frac{\beta}{x_i}\left(p_i(x)-p_i(x)^2\right)-1\\
\end{align*}

And for $j\ne i$:
\begin{align*}
    \frac{\partial G_i(x)}{\partial x_j}=-(1-r)\frac{x_i^\beta}{\left(\sum_k x_k^\beta\right)^2}x_j^{\beta-1}\beta=-(1-r)\frac{\beta}{x_i}p_i(x)^2\left(\frac{x_j}{x_i}\right)^{\beta-1}
\end{align*}

Thus:
\begin{align*}
  \frac{\partial G_i(x)}{\partial x_i}- \frac{\partial G_i(x)}{\partial x_j}&=(1-r)\frac{\beta}{x_i}\left(p_i(x)+p_i(x)^2\left(\left(\frac{x_j}{x_i}\right)^{\beta-1}-1\right)\right)-1\\
    &=(1-r)\beta\frac{p_i(x)}{x_i}\left(1+p_i(x)\left(\left(\frac{x_j}{x_i}\right)^{\beta-1}-1\right)\right)-1
\end{align*}

First, note that this condition does not depend on $\gamma$, but the position of fixed points does so of course. Moreover, it suffices to only take the richest agent for $j$ in  (\ref{eq: stability}) since $\frac{\partial G_i(x)}{\partial x_j}$ is monotone in $x_j$. Hence, we only have to check $A-1$ inequalities, such that the criterion is numerically fast.

\begin{figure}
  \centering
  \subfloat[][$r=0.4$, $\beta=2$]{\includegraphics[width=0.25\linewidth]{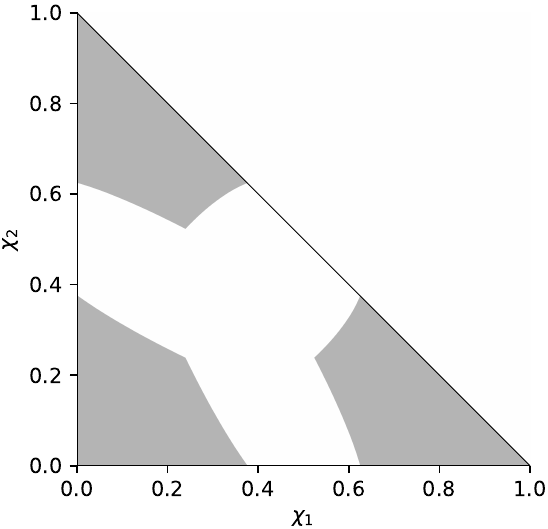}}%
  \subfloat[][$r=0.51$, $\beta=2$]{\includegraphics[width=0.25\linewidth]{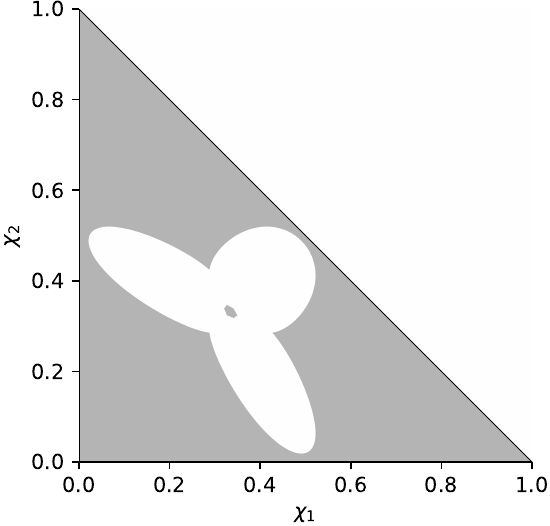}}
  \subfloat[][$r\geq 0.6$, $\beta=2$]{\includegraphics[width=0.25\linewidth]{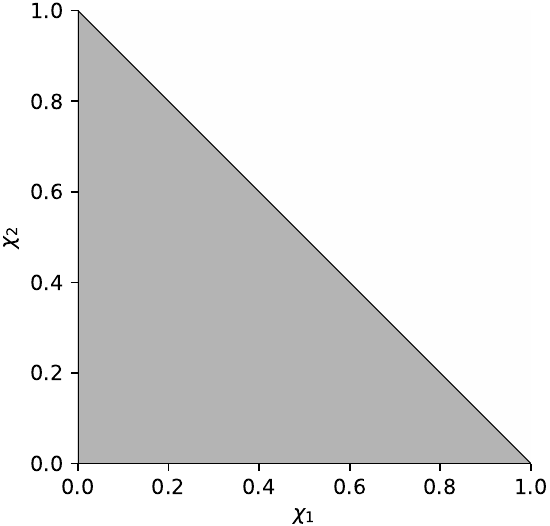}}%
  \subfloat[][$r=0.1$, $\beta=0.5$]{\includegraphics[width=0.25\linewidth]{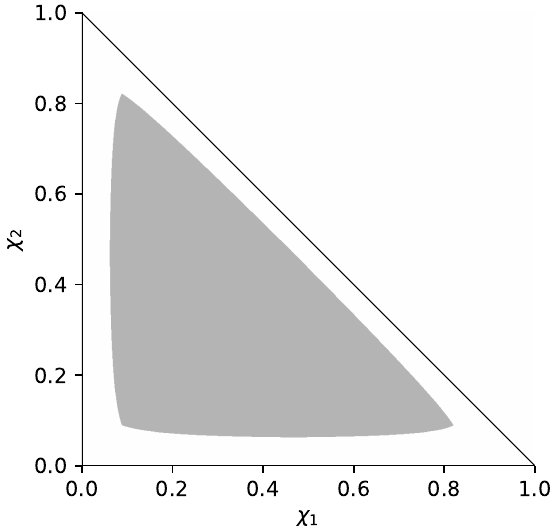}}%
  \caption{The set P for $A=3$ and different $r, \beta$.}
  \label{figure: setP}
\end{figure}

If $\beta>1$, then obviously (\ref{eq: stability}) holds for $x=e^{(k)}$ for all $k\in[A]$. Since (\ref{eq: stability}) does continuously depend on $x$, all fixed points, which are close to a corner of the simplex, are stable. On the other hand, the point $x=\left(\frac{1}{\# S}\mathds{1}_{i\in S}\right)_{i\in[A]}$, where the total wealth is shared among agents $S\subset[A]$ with $\# S>1$, fulfills (\ref{eq: stability}) only if 
$$r>\frac{\beta-1}{\beta}\,.$$
Hence, the set $P\coloneqq\{x\in\Delta_{A-1}\colon (\ref{eq: stability})\text{ holds}\}$ provides the region in $\Delta_{A-1}$ where stable fixed points can exist. It is increasing in $r$, does always contain the corners of the simplex and contains the middle point of the simplex for $r>\frac{\beta-1}{\beta}$. Figure \ref{figure: setP} illustrates this set $P$. Therefore it is plausible that for symmetric wage $\gamma=\left(\frac1A,\ldots,\frac1A\right)$ the critical labor share is  $r_c=\frac{\beta-1}{\beta}$ which is consistent with our considerations of the $A=2$ case in Section \ref{sec: A=2}. Remarkably, this expression does not depend on $A$.

For $\beta=1$, the non-linear part $G_0$ \eqref{eq: fieldG0} vanishes and the total field is $G_i (x)=r(\gamma_i -x_i )$. Therefore we have $\frac{\partial G_i(x)}{\partial x_i}- \frac{\partial G_i(x)}{\partial x_j}=-r\leq 0$, the region $P=\Delta_{A-1}$ and the unique fixed point $x=\gamma$ is stable. 
For $\beta<1$, the symmetric point $x=\left(\frac1A,\ldots,\frac1A\right)_{i\in[A]}$ fulfills (\ref{eq: stability}) for all $r\ge0$, but $P$ does not contain any boundary point of the simplex $\Delta_{A-1}$ (Figure \ref{figure: setP} (d)). Thus, any stable point must be located in the interior of the simplex.

\begin{figure}
  \centering
  \subfloat[][$\beta=2,\, x_3+x_4=0.2$]{\includegraphics[width=0.33\linewidth]{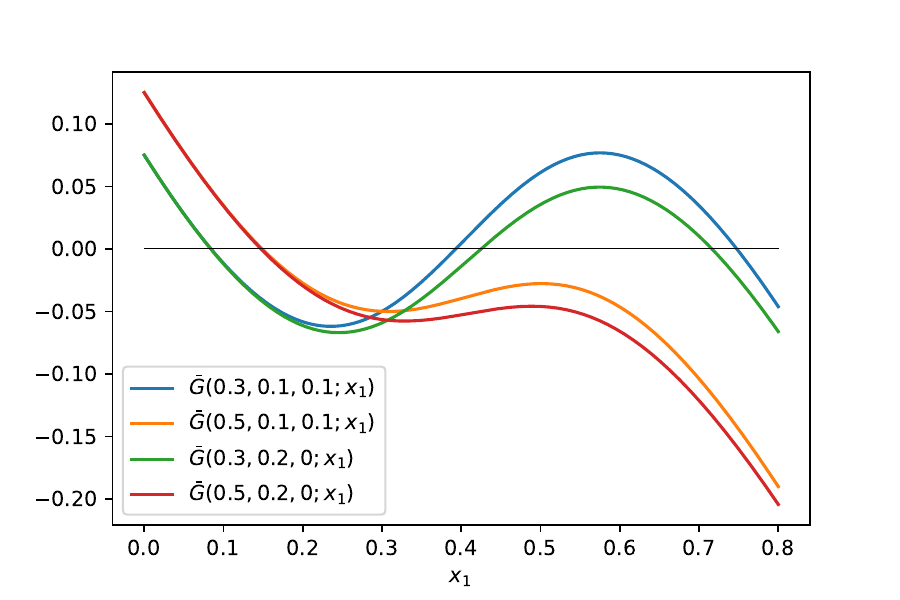}}%
  \subfloat[][$\beta=2,\, x_3+x_4=0.6$]{\includegraphics[width=0.33\linewidth]{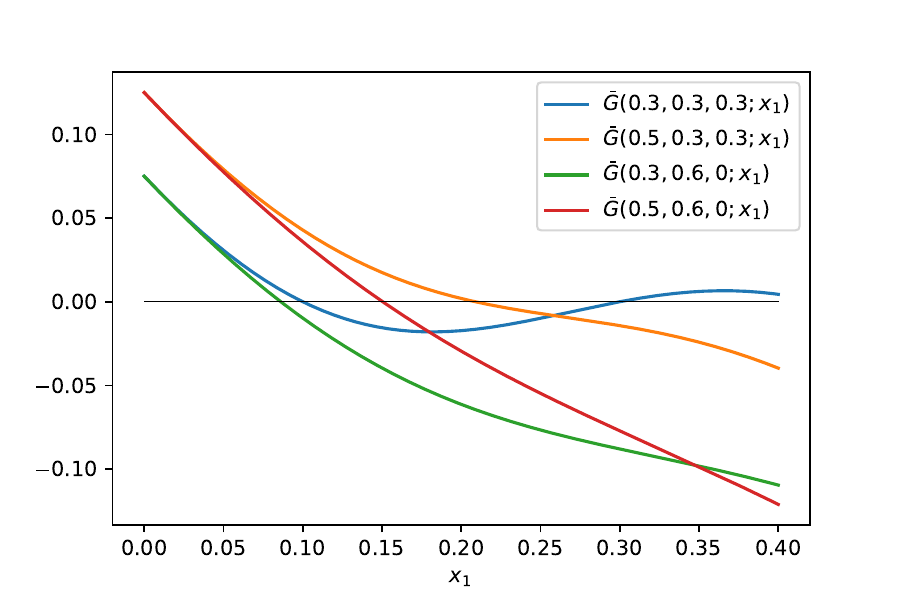}}
  \subfloat[][$\beta=0.5,\, x_3+x_4=0.4$]{\includegraphics[width=0.33\linewidth]{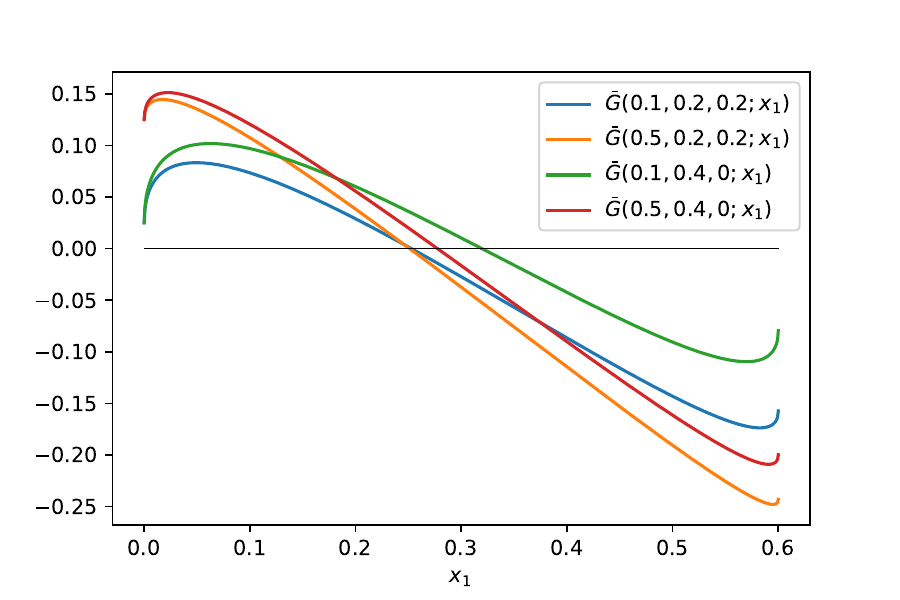}}%
  \caption{The mapping $\bar G(r, x_3, x_4; x_1)\coloneqq (1-r)\frac{x_1^\beta}{x_1^\beta+(1-x_1-x_3-x_4)^\beta+x_3^\beta+x_4}+0.25 r-x_1$ for various parameterizations.}
  \label{figure: Gheuristic}
\end{figure}

Finally, we add another heuristic to explain, why stable fixed points are either close to the vertices of the simplex or there is only one stable fixed point. For that, we take two agents $i\ne j$ and fix the shares of the others. Define $c\coloneqq\sum_{k\in[A]\atop k\ne i,\,k\ne j}x_k$ and $d\coloneqq\sum_{k\in[A]\atop k\ne i,\,k\ne j}x_k^\beta$. Then we consider the field $G$ while distributing the remaining $1-c$ share:
\begin{equation*}
    x_i\mapsto (1-r)\frac{x_i^\beta}{x_i^\beta+(1-c-x_i)^\beta+d}+r\gamma_i-x_i
\end{equation*}
Via plotting\ (Figure \ref{figure: Gheuristic}), one can easily see that for $\beta\le1$ or $c, r$ large enough,  this function has only one fixed point. Otherwise, there are three fixed points, but the middle one is unstable. Hence, stable fixed points are either unique and in the middle of the simplex ($r>r'_c$), or random and close to the corners of the simplex ($r<r'_c$). In the latter case, there are unstable fixed points and saddle points away from the corners.

\section{A functional law of large numbers for the dynamic}\label{sec: llnTax}

In complete analogy to the $r=0$ case in \cite{wir}, it is possible to describe the process in the limit for large initial values via a functional law of large numbers (LLN), where the limiting process is deterministic and given by the solution of an ordinary differential equation (ODE).

The key for the proof will be once again the Doob decomposition (\ref{eq: Doob}), where we can find separate limits for the predictable part $H$ and the martingale part $M$. In order to emphasize the dependence of the initial market size, we will write $\chi^{(N)}=\chi,\,X^{(N)}=X$ and $H^{(N)}\coloneqq H,\, M^{(N)}\coloneqq M$ in the following, where $N\coloneqq X_1^{(N)}(0)+\ldots+X_A{(N)}(0)$. We keep the initial wealth distribution $\chi(0)\in\Delta_{A-1}$ fixed. For simplicity, we assume homogeneous feedback $F_i(k)=\alpha_ik^\beta, \alpha_i>0, \beta\in\R$ for this section, such that $G(n, x)=G(x)$ and $p(n, x)=p(x)$ do not depend on the market size $n$. Assuming sufficiently fast uniform convergence of $G(n, (\cdot))$ towards a field $G$, the results can be extended to non-homogeneous feedback like in \cite{wir}. We now formulate our LLN.

\begin{theorem}\label{thm: lln_tax}
    Denote by $\tilde Z\colon[0, \infty)\to\Delta_{A-1}$ the solution of the ODE
    \begin{equation}\label{eq: ODE}
        \frac{d}{dt}\tilde Z(t)=\frac{G(Z(t))}{1+t}\quad\text{with}\quad \tilde Z(0)=\chi(0)
    \end{equation}
    and define the sequence of processes in $\Delta_{A-1}$
    $$\left(\tilde Z^{(N)}\right)_N\coloneqq\left(\tilde Z^{(N)}(t)\colon t\ge0\right)_N\coloneqq\left(\chi^{(N)}(\lfloor Nt\rfloor)\colon t\ge0\right)_N\,.$$
    Then: $\tilde Z^{(N)}$ converges to $\tilde Z$ weakly on the Skorochod space $\mathbb{D}([0, \infty)m, \Delta_{A-1})$.
\end{theorem}

Note that $G$ is locally Lipschitz, such that $\tilde Z$ is unique and well-defined. An important implication of our LLN is the following: It is possible that $\tilde Z$ does converge towards a saddle point of $G$, although our process $\chi^{(N)}$ cannot converge to saddle points due to noise. Hence, the process $\chi^{(N)}$ can be stuck in saddle points for quite a long time, i.e. it takes more than $O(N)$ time to escape from a saddle point.

The proof for the $r=0$ case in \cite{wir} can be transferred one-to-one to $r>0$, since all crucial properties do also hold in the general case. These are in particular:
\begin{enumerate}
    \item $\|\tilde Z^{(N)}(t)-\tilde Z^{(N)}(s)\|\le\frac{N|t-s|+1}{N}$ holds for $t, s\ge0$ implying tightness of the sequence $\tilde Z^{(N)}$.
    \item The the increments $\xi^{(N)}_i(n)$ are centered and uniformly bounded. Moreover, $\xi^{(N)}_i(n)$ and $\xi^{(N)}_j(m)$ are uncorrelated for $n\ne m$. Hence, Doob's inequality yields that $\left(M^{(N)}(\lfloor Nt\rfloor)\colon t\ge0\right)$ converges to zero for $N\to\infty$ (weakly on $\mathbb{D}([0, \infty), \Delta_{A-1})$).
    \item Riemann approximation of the integral yields
    $$H^{(N)}(\lfloor Nt\rfloor)=\sum_{k=0}^{\lfloor Nt\rfloor-1}\frac1N\cdot\frac{G\left(\chi^{(N)}(N\cdot\frac k N)\right)}{1+\frac k N + \frac 1N}\xrightarrow{N\to\infty}\int_0^t\frac{G(\tilde Z(t))}{1+u}du=\tilde Z(t)-\tilde Z(0),$$
    which completes the proof.
\end{enumerate}

Thus, the process behaves almost deterministic for large initial values. 

\section{Supplementary figures}

\begin{figure}[ht]
  \centering
  \subfloat{\includegraphics[width=0.8\linewidth]{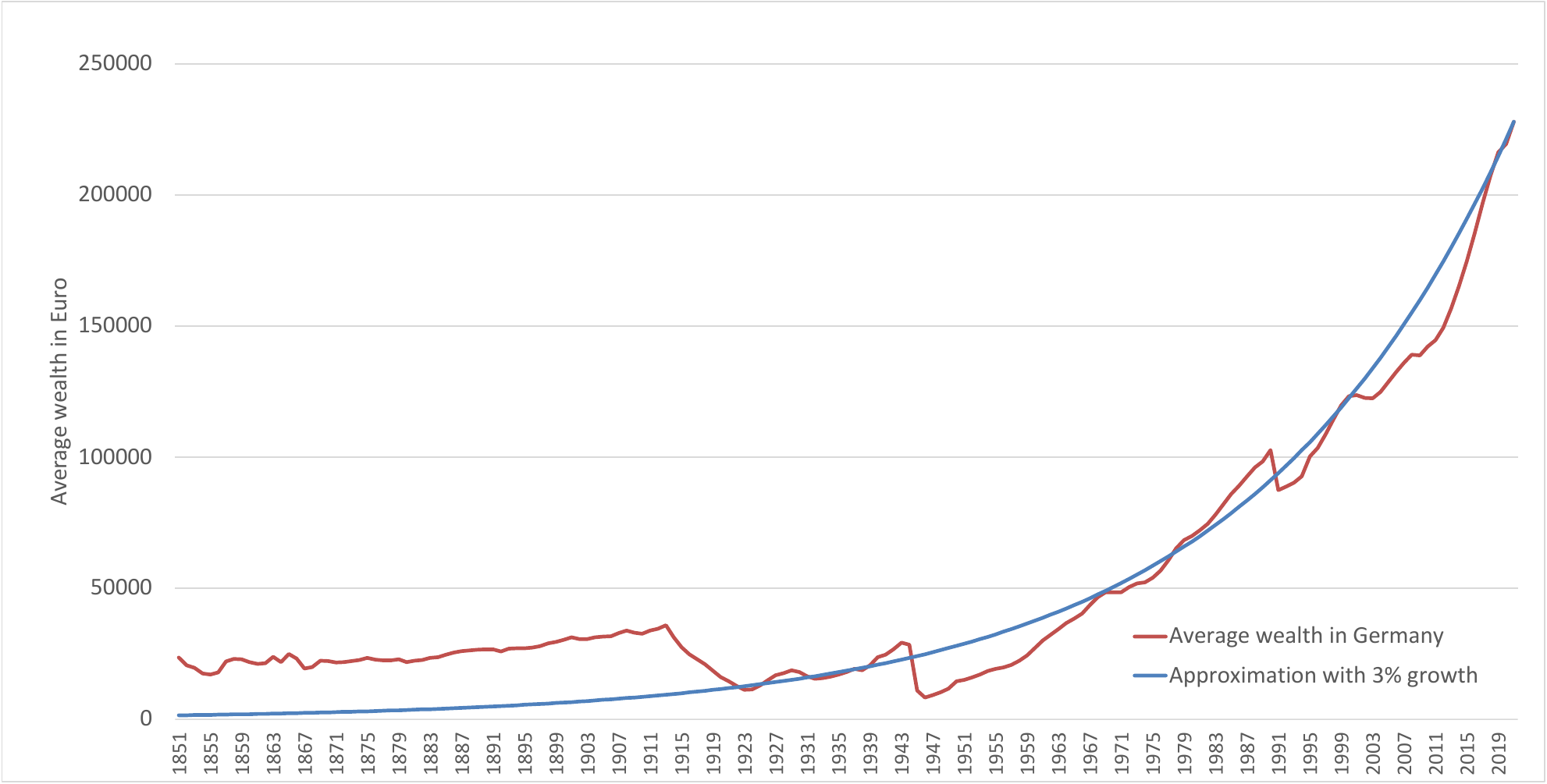}}%
  \caption{Timeline of average personal wealth in Germany according to \cite{wid} compared to constant growth at rate 3\%. This justifies the parameter choice $\mu=0.03$ in the time scaling \eqref{eq: realtime} in Section \ref{sec: future}.}
  \label{figure: wealth_timeline}
\end{figure}

\begin{figure}[ht]
  \centering
  \subfloat[]{\includegraphics[width=0.5\linewidth]{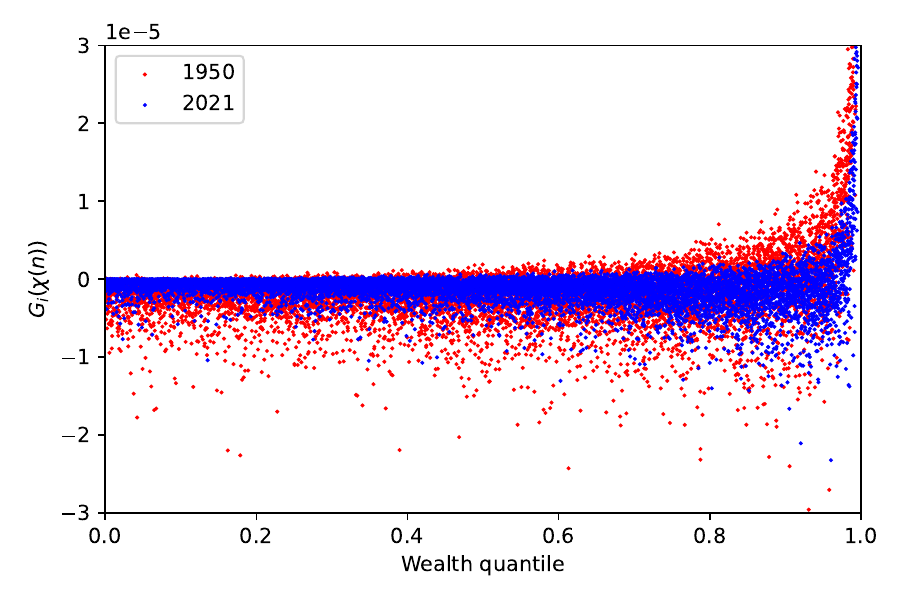}}%
  \subfloat[]{\includegraphics[width=0.5\linewidth]{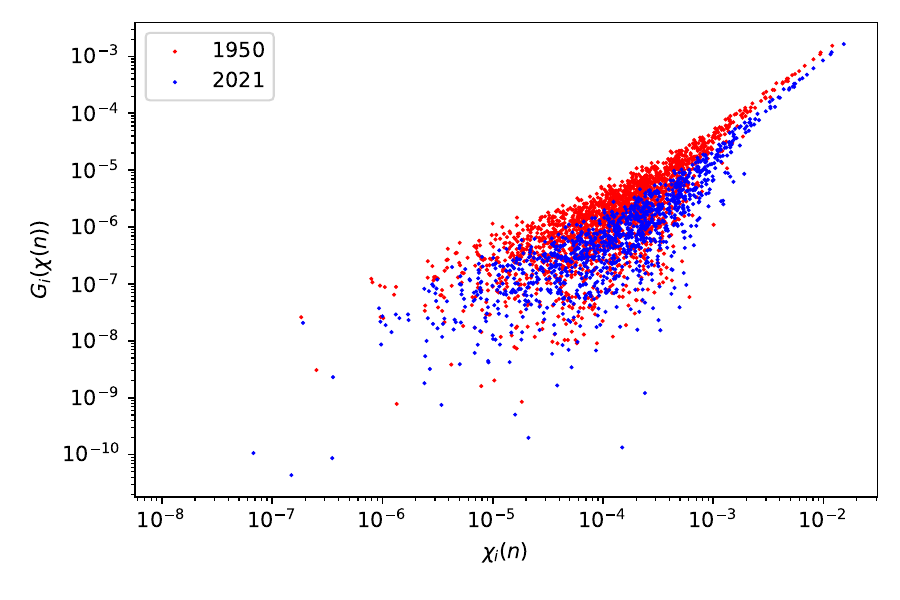}}%
  \caption{Scatter plot of $G_i(\chi(n))$ against wealth quantiles (a) and against $\chi_i(n)$ for positive $G_i$ only (b), for the years 1950 and 2021 in our simulation with symmetric initial condition. Years were assigned via (\ref{eq: realtime}) assuming a constant growth rate of $\mu=0.03$ per year. The entries of the field $G$ slowly tend to zero asymptotically and this happens faster for poorer agents (cf.\ Section \ref{sec: future}).}
  \label{figure: XGplot}
\end{figure}

\end{document}